\documentclass[11pt,a4paper,reqno]{amsart}
\usepackage{epsfig,amssymb,latexsym,amscd,amsmath,amsthm,stmaryrd, bussproofs}
\usepackage{verbatim}
\usepackage[all,2cell]{xy}
\xyoption{v2}
\UseAllTwocells
\usepackage{draftwatermark}
\SetWatermarkLightness{0.95}
\SetWatermarkFontSize{5cm}

\theoremstyle{plain}

\newtheorem{teo}{Theorem}[section]
\newtheorem{theo}[teo]{Theorem}
\newtheorem{coro}[teo]{Corollary}
\newtheorem{lema}[teo]{Lemma}

\theoremstyle{remark}

\theoremstyle{definition}

\newtheorem{defi}[teo]{Definition}
\newtheorem{obse}[teo]{Observation}

\newtheorem{question}[teo]{Question}

\usepackage[bbgreekl]{mathbbol}
\usepackage[sans]{dsfont}
\usepackage{yfonts}
\usepackage{stmaryrd}
\usepackage{euscript}
\usepackage{mathrsfs}
\usepackage{txfonts}
\usepackage[all]{xy}

\renewcommand{\Perp}{\bot\!\!\!\bot}

\newcommand{\N}{\mathds{N}}

\renewcommand{\int}{\operatorname{int}}

\newcommand{\ssi}{\Leftrightarrow} 
\newcommand{\ssucc}{\succ~\!\!\!\!\succ}

\newcommand{\sqleq}[2]{{#1} \sqsubseteq {#2}}

\newcommand{\koca}{\ensuremath{\mathcal{{}^KOCA}}}

\begin{document} 
\begin{abstract} Besides recalling the basic definitions of
  Realizability Lattices, Abstract Krivine Structures, Ordered
  Combinatory Algebras and Tripos and reviewing its relationships, we
  propose a new foundational framework for realizability. Motivated by
  Streicher's paper \emph{Krivine's Classical Realizability from a
    Categorical Perspective} \cite{kn:streicher}, we define the concept
  of \emph{Krivine's Ordered Combinatory Algebras} (\koca) as a common
  platform that is strong enough to do both: categorical and computational
  semantics. The $\mathcal{OCA}$s produced by Streicher from
  $\mathcal{AKS}$s in \cite{kn:streicher} are particular cases of
  \koca s. 
\end{abstract}
\title[A report on realizability]
{A report on realizability}
\author{Walter Ferrer Santos}
\address{Facultad de Ciencias\\Universidad de la Rep\'ublica\\
Igu\'a 4225\\11400. Montevideo\\Uruguay\\}
\email{wrferrer@cmat.edu.uy}
\author{Mauricio Guillermo}
\address{Facultad de Ingenier\'ia\\Universidad de la
  Rep\'ublica\\ J. Herrera y Reissig 565 \\ 11300. Montevideo
  \\ Uruguay\\} 
\email{mguille@fing.edu.uy}
\author{Octavio Malherbe}
\address{Facultad de Ingenier\'ia\\Universidad de la
  Rep\'ublica\\ J. Herrera y Reissig 565 \\ 11300. Montevideo
  \\ Uruguay\\} 
\email{malherbe@fing.edu.uy}
\thanks{The  authors would like to thank Csic-UDELAR and Conicyt-MEC
  for their partial support}
\today

\maketitle

\section{Introduction} 

In this report we revisit the important construction presented in the
paper: \emph{Krivine's Classical Realizability from a Categorical
  Perspective} by Thomas Streicher --see \cite{kn:streicher}--. 

As the results of Streicher's paper are the basis of our presentation
as well as of our contributions, we cite its Introduction in some
length. 

Thereat, the author states: \emph{In a sequence of papers
  \emph{(\cite{kn:kr2001};\cite{kn:kr2003}; \cite{kn:kr2009})}
  J.-L. Krivine has introduced his notion of Classical Realizability
  for classical second order logic and Zermelo-Fraenkel set
  theory. Moreover, in more recent work \emph{(\cite{kn:kr2008})} he
  has considered forcing constructions on top of it with the ultimate
  aim of providing a realizability interpretation for the axiom of
  choice. The aim of this paper is to show how Krivine's classical
  realizability can be understood as an instance of the categorical
  approach to realizability as started by Martin Hyland in
  \emph{(\cite{kn:hyland})} and described in detail in
  \emph{(\cite{kn:vanOosbook})}}. 

Later he mentions that the main purpose of his construction, is to:
(c.f. \cite{kn:streicher}) \emph{Introduce a notion of “abstract
  Krivine structure” (aks) and show how to construct a classical
  realizability model for each such aks \emph{[$\cdots$ and]} show how
  any aks $A$ gives rise to an order combinatory algebra (oca) with a
  filter of distinguished truth values which induces a tripos (see
  \emph{(\cite{kn:vanOosbook}; \cite{kn:hofstra2006})} for explanation
  of these notions) which also gives rise to a model of ZF}. 

In this report, in Sections \ref{section:two}, \ref{section:three} and
\ref{section:four}, we start with a recapitulation of the main
constructions of Streicher introducing the concept of $\mathcal {AKS}$
--the Abstract Krivine structures mentioned before-- in a modular step
by step manner, that we hope makes the subject easier to digest.  

In Sections \ref{section:five}, \ref{section:six} and
\ref{section:seven}, besides recalling the definition of combinatory
algebra and ordered combinatory algebra, we introduce the notion of
adjunctor, that is an element $\operatorname{e}$ of the algebra that
(if $\circ$ is the application and $\rightarrow$ the implication of
the algebra) guarantees that: for all $a,b,c \in A$, if $a\circ b \leq
c$, then $\operatorname{e}\circ\ a \leq (b \rightarrow c)$.  We also
show that an Abstract Krivine Structure in the sense of
\cite{kn:streicher}, produces an ordered combinatory algebra with
application, implication and adjunctor. 

In Section \ref{section:eight}, we show that --with the addition of a
completeness condition with respect to the $\operatorname{inf}$ of
arbitrary subsets to the ordered combinatory algebras considered
above-- we can induce a tripos \emph{directly} from the algebra, with
no need to first walk back to the --a priori richer-- abstract Krivine
structure.  

In Section \ref{section:nine}, we show that we can define
Realizability for high order languages in the class of
$\mathcal{OCA}$s considered in the above 
section. In particular this means that we can define Realizability for
high order arithmetics. 
In conclusion in this set up we can do both 
semantics: computational and categorical.

We would like to thank Jonas Frey and Alexandre Miquel, for sharing
with us their deep expertise on the subject, when visiting Uruguay in
2013.  

In a joint paper that is currently in preparation, more thorough
results of this collaboration will be presented.       

\section{A basic set theoretical construction: Realizability Lattices.}
\label{section:two}
\newcounter{xcounter}
\begin{list}{\Large{\bf{\arabic{xcounter}.}}}{\usecounter{xcounter}}
\item We consider the following set theoretical data. 
\begin{defi} A realizability lattice --abbreviated as $\mathcal
  {RL}$-- is a triple $(\Lambda,\Pi,\Perp)$ where $\Lambda$ and $\Pi$
  are sets and $\Perp \subseteq\Lambda \times \Pi$ is a subset. The
  elements of $\Lambda$ are called \emph{terms} and the elements of
  $\Pi$ are called \emph{stacks}. 
\begin{enumerate}
\item If $t \star \pi \in \Perp$, we write that $t \perp \pi$ and say that $t$ is perpendicular to $\pi$ or that $t$ 
realizes $\{\pi\}$.
\item Given $P \subseteq\Pi$ and $L \subseteq\Lambda$, we
  define \[{}^\perp P=\{t \in \Lambda: t \perp \pi\,,\, \forall \pi
  \in P\}\subseteq\Lambda \quad,\quad L^\perp=\{\pi \in \Pi: t \perp
  \pi\,,\, \forall t \in L\} \subseteq\Pi.\]  
\item If $t \in {}^\perp P$, we say that $t$ realizes $P$ and write $t
  \models P$. In other words $t$ realizes $P$ if and only if $t \perp
  \pi$ for all $\pi \in P$.   
\end{enumerate}
\end{defi}
\item The following definitions can be established for an $\mathcal{RL}$. 
\begin{defi} Given $(\Lambda,\Pi,\Perp)$ an $\mathcal{RL}$, 
we define a pair of maps:
\begin{align*}
(\quad)^{\perp}:\mathcal P(\Lambda)&\xrightarrow{\hspace*{3cm}} \mathcal P(\Pi)\\
\Lambda \supseteq L &\xrightarrow{\hspace*{1.5cm}} L^{\perp}=\{\pi \in
\Pi|\,\, \forall t \in L, t \star \pi \in \Perp\}=\{\pi \in \Pi|\,\, L
\times \{\pi\} \subseteq \Perp\}\subseteq \Pi; 
\end{align*}
\begin{align*}
{}^{\perp}(\quad):\mathcal P(\Pi)&\xrightarrow{\hspace*{3cm}} \mathcal P(\Lambda)\\
\Pi \supseteq P &\xrightarrow{\hspace*{1.5cm}} {}^{\perp}P=\{t \in
\Lambda|\,\, \forall \pi \in P, t \star \pi \in \Perp\}=\{t \in
\Lambda|\,\, \{t\} \times P \subseteq \Perp\}\subseteq \Lambda. 
\end{align*}
The pairs of $\Lambda\times\Pi$ are called \emph{processes} and it is
customary to denote the
process $(t, \pi)$ as $t\star\pi$. 
\end{defi} 
\begin{obse}\label{obse:initialrl}In the notations above for an $\mathcal{RL}$ one has that:
\begin{enumerate}
\item The maps $L \rightarrow L^{\perp}$ and $P \rightarrow {}^{\perp}P$ are antimonotone with respect to the order given by the inclusion of sets and ${}^{\perp}\emptyset=\Lambda$ and $\emptyset{}^{\perp}=\Pi$. 
\item Let us consider the behaviour of the operators $(\quad)^{\perp}$ and ${}^{\perp}(\quad)$ with respect to the lattice structure of the domain and codomain. We have that for $P_i \subseteq\Pi, {i \in I}$ and $L_i \subseteq\Lambda, {i \in I}$: 
 \[{}^{\perp}\big(\bigcap_{i \in I} P_i\big) \supseteq \bigcup_{i \in I} {}^{\perp}P_i\,,\,{}^{\perp}\big(\bigcup_{i \in I} P_i\big) = \bigcap_{i \in I} {}^{\perp}P_i;\]\[\big(\bigcap_{i \in I} L_i\big)^{\perp} \supseteq \bigcup_{i \in I} L_i^{\perp}\,,\,\big(\bigcup_{i \in I} L_i\big)^{\perp} = \bigcap_{i \in I} L_i^{\perp}.\]  
\item For an arbitrary $L \in \mathcal P(\Lambda)$ and $P \in \mathcal P(\Pi)$, one has that ${}^{\perp}(L^{\perp}) \supseteq L$ and $({}^{\perp}P)^{\perp} \supseteq P$.
\item \label{item:perpnonempty} One has that $({}^{\perp}\Pi)^{\perp}= \Pi$ and ${}^{\perp}(\Lambda^{\perp})= \Lambda$. Notice that in general it may happen that ${}^\perp \Pi \neq \emptyset$ or $\Lambda^\perp \neq \emptyset$ --see later Observation \ref{obse:weakercond},\eqref{item:nonempty}. 
\item For an arbitrary $L \in \mathcal P(\Lambda)$ and $P \in \mathcal P(\Pi)$, one has that $({}^\perp(L^{\perp}))^{\perp} = L^{\perp}$ and ${}^{\perp}(({}^\perp P)^\perp) = {}^{\perp}P$.  
\end{enumerate}
\begin{proof}
The proof of the first four properties is immediate. For the fifth one, applying the $\perp$ operator in ${}^{\perp}(L^{\perp}) \supseteq L$ we obtain that  $({}^\perp(L^{\perp}))^{\perp} \subseteq  L^{\perp}$and substituting in the inequality $({}^{\perp}P)^{\perp} \supseteq P$, the subset $P$ by  $L^{\perp}$ we obtain the reverse inclusion. Similarly for subsets $P \subseteq \Pi$. 
\end{proof}
\end{obse}
\smallskip
\item In the above context, the following definition is natural.
\begin{defi} In the situation that we have an $\mathcal{RL}$ as above, we define the following sets:
\begin{align*} \mathcal P_{\perp}(\Lambda)&=\{L \subseteq \Lambda|\, {}^\perp(L^{\perp}) = L\} \subseteq \mathcal P(\Lambda),\\
\mathcal P_{\perp}(\Pi)&=\{P \subseteq \Pi|\, ({}^{\perp}P)^{\perp} = P\} \subseteq\mathcal P(\Pi)
\end{align*}
\end{defi}
Notice that the only relevant structure at this point is the
\emph{lattice} structure in the sets $P_{\perp}(\Lambda)$ and
$P_{\perp}(\Pi)$, where we take the (set theoretical) inclusion as the order and
as ``meet'' and ``join'' the intersection and union respectively
follwed by taking double perpendicularity.  
\begin{lema} In the above context of an $\mathcal {RL}$ the maps 
$(\quad)^{\perp}:\mathcal P(\Lambda)\rightarrow \mathcal P(\Pi)$ and 
${}^{\perp}(\quad):\mathcal P(\Pi)\rightarrow \mathcal P(\Lambda)$
  when restricted respectively to $\mathcal P_{\perp}(\Lambda)$ and
  $\mathcal P_{\perp}(\Pi)$ are order reversing isomorphisms inverse
  of each other. Moreover with respect to the order given by the
  inclusion, $\Lambda^\perp$ and $\Pi$\emph{;} ${}^\perp\Pi $ and
  $\Lambda $ are the minimal and maximal elements of $\mathcal
  P_{\perp}(\Pi)$ and $\mathcal P_{\perp}(\Lambda)$ respectively.  
\end{lema}
\begin{proof} This result follows immediately from the previous
  considerations and it is in fact a general result concerning a
  Galois connection. 

Indeed, it is clear that $\operatorname{Im}((\quad)^{\perp}) =
\mathcal P_{\perp}(\Pi)$ and $\operatorname{Im}({}^{\perp}(\quad)) =
\mathcal P_{\perp}(\Lambda)$.  

By the very definition of $\mathcal P_{\perp}(\Lambda)$ it is clear
that if we apply succesively the maps $(\quad)^{\perp}:\mathcal
P(\Lambda)\rightarrow \mathcal P(\Pi)$ and~${}^{\perp}(\quad):\mathcal
P(\Pi)\rightarrow \mathcal P(\Lambda)$ to $L \in \mathcal
P_{\perp}(\Lambda)$ we obtain again $L$.  
Similarly for $P \in \mathcal P_{\perp}(\Pi)$.  
\end{proof}

The following observation will be used repeatedly.

\begin{obse}\label{obse:repeatedly} The following results are valid in an $\mathcal {RL}$. Notice that the last three assertions need stronger hypothesis than the first.
\begin{enumerate}
\item If $L \in \mathcal P(\Lambda)$ and $P \in \mathcal P(\Pi)$, then $L \subseteq {}^\perp P$ if and only if $P \subseteq L^\perp$. 

\item If $L \in \mathcal P_\perp(\Lambda)$ and $P \in \mathcal P(\Pi)$, then $L^\perp \subseteq P$ implies that ${}^\perp P \subseteq L$.
\item If $L \in \mathcal P(\Lambda)$ and $P \in \mathcal P_\perp(\Pi)$, then ${}^\perp P \subseteq L$ implies that $L^\perp \subseteq P$.
\item If $L \in \mathcal P_\perp(\Lambda)$ and $P \in \mathcal P_\perp(\Pi)$, then $L^\perp \subseteq P$ if and only if ${}^\perp P \subseteq L$.
\end{enumerate} 

As in the general situation of a Galois connection, the above conditions
(1) and (4) can be read as  adjunction relations between the functors ${}^\perp(-)$ and $(-)^\perp$ in the adequate domain and codomain.  
\end{obse}
\section{The push map in a realizabilty lattice}
\label{section:three}
\item In this section we add what we call \emph{a push map} to a
  realizability lattice, with which we can add the first elements of a
  \emph{calculus} to our structure.  
\begin{defi} A map $(t,\pi) \mapsto t . \pi : \Lambda \times \Pi
  \rightarrow \Pi$ defined in a realizability lattice
  $(\Lambda,\Pi,\Perp)$, will be called a \emph{push} map and denoted
  as $\operatorname{push}(t,\pi)=t . \pi$. In that case we say that
  the realizability lattice is endowed with a push map.   
\end{defi} 

\begin{defi}\label{defi:pushconductor} For an $\mathcal{RL}$ with a push, for $L \subseteq \Lambda$ and $P \subseteq \Pi$ we define:
\[L \leadsto P=\{\pi \in \Pi: L.\pi \subseteq P\} \subseteq \Pi\quad \text{right conductor of $L$ into $P$}.\]  
\end{defi}

Notice that:
\[L \leadsto P=\bigcup \{Q \subseteq \Pi: L.Q \subseteq P\}.\]

We can use the push map in order to define a map:
\[(L,P)\mapsto L.P:\mathcal P(\Lambda) \times \mathcal P(\Pi) \rightarrow \mathcal P(\Pi),\] 
that combined with the operators $(\quad)^\perp$ and ${}^\perp(\quad)$ yields natural binary operations in $\mathcal P_{\perp}(\Lambda)$ and $\mathcal P_{\perp}(\Pi)$.

\begin{obse} We can interpret the maps in Definition \ref{defi:pushconductor} as follows. Consider $L \subseteq \Lambda$ and define $a_L,m_L:\mathcal P(\Pi) \rightarrow \mathcal P(\Pi)$ as $a_L(P)=L \leadsto P$ and $m_L(P)=L.P$\footnote{In principle, the maps defined above  are not internal maps in the corresponding $\mathcal P_\perp$s.}.  In this notation the following ``adjunction relations'' holds:  For all $P,Q \subseteq \Pi$:
\[m_L(Q) \subseteq P \Leftrightarrow Q \subseteq a_L(P).\]
\end{obse}

\begin{defi} \label{defi:maps}We define the following binary operations in $\mathcal P_\perp(\Pi)$.
Let $P,Q \in \mathcal P_{\perp}(\Pi)$:
\begin{enumerate}
\medskip
\medskip
\item \hspace*{1cm}\label{item:maps1} $P\circ Q=({}^\perp\{\pi \in \Pi: {}^\perp Q .\,\pi \subseteq P \})^\perp= ({}^\perp({}^\perp Q\leadsto P))^{\perp} \in \mathcal P_\perp(\Pi)$. 
\medskip
\medskip
\item \hspace*{1cm}\label{item:maps2} $P \rightarrow Q=
  ({}^\perp\operatorname{push}({}^\perp
  P,Q))^{\perp}=({}^\perp({}^\perp P \cdot Q))^{\perp} \in \mathcal
  P_{\perp}(\Pi)$. 
\end{enumerate}
\end{defi}
\begin{obse}
\begin{enumerate}
\item Observe that in accordance to the above Definition
  \ref{defi:maps}, \eqref{item:maps1}, we have that for $P,Q \in
  \mathcal P_\perp(\Pi)$: 
\[ P \subseteq ({}^\perp Q \cdot P)\circ Q.\footnote{Notice the slight
  abuse of notation in this formula commited by applying the $\circ$
  operation in a situation in which one of the sets is not invariant
  by double perpendicularity.}\] 

\item Notice that: $P\circ Q=({}^\perp\{\pi \in \Pi: {}^\perp P
  \subseteq {}^\perp({}^\perp Q .\,\pi)\})^{\perp}$. 

\item From the definition of $P \rightarrow Q$, we deduce that ${}^\perp(P \rightarrow Q)= {}^\perp({}^\perp P.Q)$.  
\end{enumerate}
\end{obse}

\item From the above Definition \ref{defi:maps}, we can deduce a
  crucial ``half adjunction property'' relating the operations $\circ$
  and $\rightarrow$ in $\mathcal P_\perp(\Pi)$.  
\begin{theo}\label{theo:adjunction}\emph{[Half adjunction property]}
  Assume that $P,Q,R \in \mathcal P_{\perp}(\Pi)$. If $Q \rightarrow R
  \subseteq P$, then $R \subseteq P \circ Q$.  
\end{theo}
\begin{proof} The inclusion $Q \rightarrow R \subseteq P$ means that
  $({}^\perp({}^{\perp}Q \cdot R))^{\perp} \subseteq P$ and this is
  equivalent to ${}^{\perp}Q \cdot R \subseteq P$. Now, this implies
  that $R \subseteq \{\pi \in \Pi: {}^\perp Q .\,\pi \subseteq P\}$
  that implies that $R \subseteq P \circ Q$.  
\end{proof}
\begin{obse}\label{obse:preadjunction} 
\begin{enumerate}
\item We have used the following elementary fact: if $P,Q,R \in
  \mathcal P_\perp(\Pi)$. Then, $P \supseteq {}^\perp Q \cdot R$ if
  and only if  $ \{\pi \in \Pi: P \supseteq {}^\perp Q .\,\pi\}
  \supseteq R$. 
\item From the above comment it follows that if $({}^\perp\{\pi \in
  \Pi: P \supseteq {}^\perp Q .\,\pi\})^{\perp}=\{\pi \in \Pi: P
  \supseteq {}^\perp Q .\,\pi\}$ --i.e. if $\{\pi \in \Pi: P \supseteq
            {}^\perp Q .\,\pi\} \in \mathcal P_\perp(\Pi)$-- then the
            conditions $P \circ Q \supseteq R$ and $P \supseteq Q
            \rightarrow R$ are equivalent. 
\item Along the proof of Theorem \ref{theo:adjunction} we obtained the
  following fact: the inclusion $Q \rightarrow R \subseteq P$ is
  equivalent to  $({}^\perp({}^{\perp}Q \cdot R))^{\perp} \subseteq P$
  that is equivalent to ${}^{\perp}Q \cdot R \subseteq P$. 
\end{enumerate}
\end{obse}
Using the above adjunction result --Theorem \ref{theo:adjunction}--in the case that $P=Q \rightarrow R$ we obtain the following Corollary.
\begin{coro}\label{coro:adjunction} For all $R,Q \in \mathcal P_\perp(\Pi)$, we have that $R \subseteq (Q \rightarrow R)\circ Q$.
\end{coro}
\item It is important to remark that in fact, the operations $\circ$
  and $\rightarrow$ are not independent. Their close relationship is
  illustrated in the theorem that follows .  



\section{Abstract Krivine Structures.}
\label{section:four}
\item In this section we complete the definition of a calculus in a
  realizability lattice to obtain the concept of  \emph{pre--Abstract
    Krivine Structure} abbreviated as $\mathcal {PAKS}$. For that, we
  introduce the usual application map for terms, a save map from
  stacks to terms, the combinators
  $\operatorname{K},\operatorname{S}$, and a distinguished term
  $\operatorname{cc}$ that is a realizer of Peirce's law.   
\begin{defi} \label{defi:aks}A \emph{pre--Abstract Krivine Structure} consists of the following elements:
\begin{enumerate}
\item \label{item:aksdecuple} A
  nonuple \[(\Lambda,\Pi,\Perp,\operatorname{app},\operatorname{save},\operatorname{push},
  \operatorname{K},\operatorname{S},\operatorname{cc}),\] 
where: 
\medskip
\begin{enumerate}
\item $(\Lambda, \Pi, \Perp)$ is an  $\mathcal {RL}$.
\item $\operatorname{app}:\Lambda \times \Lambda \rightarrow \Lambda$ is a function: $(t,u) \mapsto \operatorname{app}(t,u)=tu$.
\item $\operatorname{save}:\Pi \rightarrow \Lambda$ is a function: $\pi \mapsto \operatorname{save}(\pi)=k_\pi$.
\item $\operatorname{push}:\Lambda \times \Pi \rightarrow \Pi$ is a
  function. We abbreviate $(t,\pi) \mapsto \operatorname{push}(t,\pi)=t . \pi$.
\item $\operatorname{K},\operatorname{S},\operatorname{cc} \in \Lambda$ are distinguished elements.  
\end{enumerate}
The elements of $\Lambda \times \Pi$ are called \emph{processes} and
we denote the process $(t, \pi)$ as $t\star\pi$. 
\item \label{item:aksaxioms} The above elements are subject to the following axioms.
\medskip
\newcounter{qcounter}
\begin{list}{(S\arabic{qcounter})}{\usecounter{qcounter}}
\item If $t \star s.\pi \in \Perp$,  then $ts \star \pi \in \Perp$
  --in the case that the converse holds, i.e. if  $ts \star \pi \in
  \Perp$ implies that $t \star s \cdot \pi \in \Perp$, we say that the
  given $\mathcal {PAKS}$ is \emph{strong}.

\item If $t \star \pi \in \Perp$, then for all $s \in \Lambda$ we have that $\operatorname{K} \star \, t \cdot s \cdot \pi \in \Perp$. 
\item If $tu (su) \star \pi \in \Perp$, then $\operatorname{S} \star \, t \cdot s \cdot u \cdot \pi \in \Perp$.
\item If $t \star k_\pi \cdot \pi \in \Perp$, then $\operatorname{cc} \star \, t \cdot \pi \in \Perp$.
\item If $t\star\pi \in \Perp$, then for all $\pi'\in \Pi$ we have that $k_\pi \star t \cdot \pi' \in \Perp$.
\end{list}
\end{enumerate} 
\end{defi}
\item Here and in the rest of these notes, product--like operations will --in general-- be non associative. Hence, when parenthesis are omitted it is implicit that we associate to the left. In other words: \[a_1a_2a_3=(a_1a_2)a_3\,\, \text{and in general}\,\, a_1a_2a_3\cdots a_n=(a_1a_2a_3\cdots a_{n-1})a_n.\] 
\item Notice, that besides adding the application, the save map and three distinguished terms to the structure of a \emph{realizability lattice with a push}, we have introduced five axioms that interrelate the above data and that can be divided into three groups. The first axiom interrelates the newly defined application map with the push. The second and third establishes interactions between the combinators and the push map, while the 
fourth and fifth establishes relations between the push map, the save map and the distinguished element $\operatorname{cc}$. 
\item The elements of the structure above, named as: \[save:\pi \mapsto k_\pi :\Pi \rightarrow \Lambda \quad \text{and} \quad \operatorname{cc} \in \Lambda,\] have a very special role in the sense that they make the realizability theory \emph{classical} as $\operatorname{cc}$ realizes Pierce's law. In this sense it may be convenient to introduce the following nomenclature, in the presence of the mentioned elements and the corresponding axioms (S4) and (S5), we say that the $\mathcal {PAKS}$ --and later the $\mathcal {AKS}$-- is \emph{classical}.  
\item \label{item:axiomsperp} The axioms for a $\mathcal {PAKS}$ appearing in Definition \ref{defi:aks}, \eqref{item:aksaxioms} can be formulated also as follows:
\newcounter{tcounter}
\begin{list}{(S\arabic{tcounter})}{\usecounter{tcounter}}
\item If $t \perp s \cdot \pi$,  then $ts \perp \pi$ --moreover $ts \perp \pi$ if and only if $t \perp s \cdot \pi$ in the case that the given $\mathcal {PAKS}$ is strong. 
\item If $t \perp \pi$, then for all $s \in \Lambda$ we have that $\operatorname{K} \perp t \cdot s \cdot \pi$. 
\item If $tu (su) \perp \pi$, then $\operatorname{S} \perp  t \cdot s \cdot u  \cdot \pi$.
\item If $t \perp k_\pi \cdot \pi$, then $\operatorname{cc} \perp t \cdot \pi$.
\item If $t \perp \pi$, then for all $\pi'\in \Pi$ we have that $k_\pi \perp t \cdot \pi'$. 
\end{list}
\begin{obse}\label{obse:weakercond} The following weaker consequences of the last three axioms can be deduced easily applying (S1) to (S2)\ldots(S5). In the case that the $\mathcal {PAKS}$ is strong, the conditions below are equivalent to the original ones. 
\begin{enumerate}
\item If $t \perp \pi$, then for all $s \in \Lambda$ we have that $\operatorname{K}ts \perp \pi$. 
\item If $tu (su) \perp \pi$, then $\operatorname{S}tsu \perp \pi$.
\item If $t \perp k_\pi \cdot \pi$, then $\operatorname{cc}t \perp \pi$. If the $\mathcal {PAKS}$ is strong we have that: if $tk_\pi \perp \pi$, then $\operatorname{cc}t \perp \pi$.
\item \label{item:nonempty} If $t \perp \pi$, then for all $\pi' \in \Pi$ we have that $k_\pi t \perp \pi'$. In other words if $t \perp \pi$, then $k_\pi t \in {}^\perp\Pi$.
\end{enumerate}
Observe that the last assertion exhibits a situation related to Observation \ref{obse:initialrl}, \eqref{item:perpnonempty}
\end{obse}
\item Once we have at our disposal the map $\operatorname{app}:(t,s)\mapsto ts: \Lambda \times \Lambda \rightarrow \Lambda $,  we can define the following \emph{conductor}  for $L,M \subseteq \Lambda$ --compare with the previous Definition \ref{defi:pushconductor}--:
\[L \leadsto M=\{t \in \Lambda: tL \subseteq M\} \subseteq \Lambda\quad \text{left conductor of $L$ into $M$}.\]
Notice that similarly than before, the above conductor can be characterized in the following way:
\[L \leadsto M=\bigcup \{L' \subseteq \Lambda: L'L \subseteq M.\}\]
Considering also the natural operator coming from the application
--$\operatorname{app}$--:  
\begin{eqnarray}
\nonumber (L,M) \mapsto LM:\mathcal P(\Lambda) \times \mathcal P(\Lambda) \rightarrow \mathcal P(\Lambda), 
\end{eqnarray} 
we may define the maps $a_L,m_L: \mathcal P(\Lambda) \rightarrow
\mathcal P(\Lambda)$ by $a_L(M)=L \leadsto M$ and $m_L(M)=LM$\footnote{In principle, the maps defined above  are not internal maps in the corresponding $\mathcal P_\perp$s.}. 

We have the following adjoint relationship: for all $L,M,N \subseteq \Lambda$; \[m_L(N) \subseteq M \Leftrightarrow N \subseteq a_L(M).\]
\begin{defi} \label{defi:maps2}For $P,Q \in \mathcal P_{\perp}(\Pi)$ we define the following binary operation in $\mathcal P_\perp(\Pi)$:
\[P \diamond Q= \operatorname{app}({}^\perp P,{}^\perp Q)^\perp=\big(({}^\perp P) ({}^\perp Q)\big)^{\perp} \in \mathcal P_{\perp}(\Pi).\]

\end{defi}
\begin{obse}\label{obse:simplification}
The importance of the three operations $\circ, \rightarrow$ and $\diamond$ defined in $\mathcal P_\perp(\Pi)$ can be visualized when one performs the following computations in singleton sets.
\begin{enumerate}
\item \label{item:simplcirc} Let us take $t,s \in \Lambda$. We have that: \begin{equation}\{t\}^\perp\circ{}\{s\}^\perp=\big({}^\perp\{\pi \in \Pi: r\perp \ell.\pi \,\,\forall r \perp \{t\}^\perp, \forall \ell \perp \{s\}^\perp\}\big)^\perp,
\end{equation}\label{eqn:circ1} 
and as $t \in \{t\}^\perp; s \in \{s\}^\perp$ we deduce that:
$\big({}^\perp\{\pi \in \Pi: r\perp \ell.\pi \,\,\forall r \perp \{t\}^\perp, \forall \ell \perp \{s\}^\perp\}\big)^\perp \subseteq \big({}^\perp\{\pi \in \Pi:t \perp s.\pi\}\big)^\perp$.
Then:  
 \begin{equation}\label{eqn:circ2}\{t\}^\perp\circ\{s\}^\perp \subseteq \big({}^\perp\{\pi \in \Pi:t \perp s.\pi\}\big)^\perp.\end{equation}

\item \label{item:simpldiamond}Moreover, by definition we have that $\{t\}^\perp \diamond \{s\}^\perp=\Big(\big({}^\perp(\{t\}^{\perp})\big)\big({}^\perp(\{s\}^{\perp})\big)\Big)^\perp$ and:
\begin{equation}\label{eqn:diamond1}\{t\}^\perp \diamond \{s\}^\perp=\{\pi \in \Pi: r\ell \perp \pi \,\,\forall r \perp \{t\}^\perp, \forall \ell \perp \{s\}^\perp\}.\end{equation} 
Hence: \begin{equation}\label{eqn:diamond2}\{t\}^\perp \diamond \{s\}^\perp \subseteq \{ts\}^{\perp}.\end{equation}
\item \label{item:S1diamondcirc} Next we show that there is a very close relationship between the operations $\circ, \diamond$ and the basic condition (S1) of Definition \ref{defi:aks}, \eqref{item:aksaxioms}. 

Indeed, condition (S1) implies that:
\begin{align*} \big({}^\perp\{\pi \in \Pi: r\perp \ell.\pi \,\,\forall r \perp \{t\}^\perp, \forall \ell \perp \{s\}^\perp\}\big)^\perp &\subseteq \big({}^\perp\{\pi \in \Pi: r\ell\perp \pi \,\,\forall r \perp \{t\}^\perp, \forall \ell \perp \{s\}^\perp\}\big)^\perp=\\
=\{\pi \in \Pi &: r\ell\perp \pi \,\,\forall r \perp \{t\}^\perp, \forall \ell \perp \{s\}^\perp\}.
 \end{align*}
Using the characterization of the operations appearing in \eqref{eqn:circ1} and \eqref{eqn:diamond1}, we deduce that in the presence of condition (S1) we have that for all $t,s$ \[\{t\}^\perp \circ \{s\}^\perp \subset \{t\}^\perp \diamond \{s\}^\perp.\]
\item \label{item:simplarrow} Concerning the implication we have: \[\{t\}^\perp \rightarrow Q =({}^\perp({}^\perp(\{t\}^{\perp}). Q))^{\perp}\supseteq ({}^\perp(t. Q))^{\perp},\] or equivalently:
\[{}^\perp(\{t\}^\perp \rightarrow Q) \subseteq {}^\perp(t. Q).\]
\end{enumerate}
\end{obse}

\item For future use, it is interesting to write down the basic axioms
  of a $\mathcal {PAKS}$ in terms of elements of $\mathcal
  P_\perp(\Lambda)$ and $\mathcal P_\perp(\Pi)$ and the operations
  $\circ$, $\rightarrow$, $\diamond$ and the conductors. We emphasize
  --with an eye in future use-- the formulation in terms of $\mathcal
  P_\perp(\Pi)$.   
\begin{lema}\label{lema:axiomset} The axioms of a $\mathcal {PAKS}$
  presented in Definition \ref{defi:aks}, \eqref{item:aksaxioms}
  --also appearing in an equivalent formulation in
  {\Large{\bf{\ref{item:axiomsperp}.}}}--, have the following
  consequences. Assume that $P,Q,R$, are generic elements of
  $\mathcal P_\perp(\Pi)$, and that $\operatorname{K,S,cc} \in
  \Lambda$ are as before, then:    
\newcounter{bcounter}
\begin{list}{\emph{($\mathbb
      S$\arabic{bcounter})}}{\usecounter{bcounter}} 
\item \label{item:circdiamond} Condition \emph{(S1)} in Definition
  \ref{defi:aks}, \eqref{item:aksaxioms} implies condition \emph{(1)}
  that implies condition \emph{(3)} that implies \emph{(2)}.  
\begin{enumerate} 
\item $P \circ Q \subseteq ({}^\perp P {}^\perp Q)^\perp =P \diamond
  Q$ or equivalently ${}^\perp P {}^\perp Q \subseteq {}^\perp(P \circ
  Q)$ or equivalently: if $t \perp P$ and $s \perp Q$, then $ts \perp
  P \circ Q$. 
\item $({}^\perp P \leadsto {}^\perp Q)^\perp \subseteq P \rightarrow
  Q$ or equivalently $Q \subseteq (P \rightarrow Q) \diamond P$.  
\item If ${}^\perp Q . R \subseteq P$, then ${}^\perp P {}^\perp Q
  \subseteq {}^\perp R$. Equivalently, if $Q \rightarrow R \subseteq
  P$, then $R \subseteq P \diamond Q$.  
\end{enumerate}

\item The first condition below is equivalent to condition \emph{(S2)}
  in Definition \ref{defi:aks}, \eqref{item:aksaxioms}, and the second
  is a consequence. 
\begin{enumerate}
\item For all $P,R$, we have that $\operatorname{K} \in
  {}^\perp({}^\perp P.{}^\perp R.P)$. Equivalently, for all $P
  \subseteq Q$ we have that $K \in {}^\perp({}^\perp Q.{}^\perp
  R.P)$. 
\item For all $P,R$, we have that $\operatorname{K}{}^\perp P {}^\perp
  R \subseteq {}^\perp P$. Equivalently, for all $P \subseteq Q$ we
  have that $\operatorname{K}{}^\perp Q {}^\perp R \subseteq {}^\perp
  P$ 
\end{enumerate}
\item The first condition below is equivalent to condition \emph{(S3)}
  in Definition \ref{defi:aks}, \eqref{item:aksaxioms}, and the second
  is a consequence. 
\begin{enumerate}
\item If ${}^\perp P u({}^\perp Q u) \subseteq {}^\perp R$ then
  $\operatorname{S} \in {}^\perp({}^\perp P.{}^\perp Q .u. R)$ with $u
  \in \Lambda$. 
\item If ${}^\perp P u({}^\perp Q u) \subseteq {}^\perp R$ then
  $\operatorname{S} {}^\perp P {}^\perp Q u \subseteq {}^\perp R$ with
  $u \in \Lambda$. 
\end{enumerate}
\item \label{item:classicalcc} 
\begin{enumerate}
\item Axiom \emph{(S4)} in Definition \ref{defi:aks} is equivalent to: 
$\operatorname{cc}\in
  {}^\perp\big((\operatorname{save}(P).P)\rightarrow P\big)$.  
\item Axiom \emph{(S5)} in Definition \ref{defi:aks} is equivalent to: 
$\operatorname{save}(P) \subseteq {}^\perp(P \rightarrow Q)$. 
\item Axioms \emph{(S4)} and \emph{(S5)} imply that for all $P,Q \in
  \mathcal P_\perp(\Pi)$: $\operatorname{cc} \perp ((P \rightarrow
  Q)\rightarrow P)\rightarrow P$. In other words the axioms imply that
  the term $\operatorname{cc} \in \Lambda$ realises Peirce's law. 
\end{enumerate} 
\end{list}
\end{lema}
\begin{proof}
\newcounter{vcounter}
\begin{list}{{($\mathbb S$\arabic{vcounter})}}{\usecounter{vcounter}}
\item \ \\
\vspace{-12pt}
\begin{itemize}
\item It is evident that the three formulations of
  condition (1) are 
equivalent.  
\item The two formulations of condition (2) are equivalent. Indeed,
  for arbitrary $L,M \subseteq \Lambda$ we have that $(L \leadsto M)=
  \bigcup \{N \subseteq \Lambda: NL \subseteq M\}$ and then $(L
  \leadsto M)^\perp= \bigcap \{N^\perp \subseteq \Pi: NL \subseteq
  M\}$ and then $({}^\perp P \leadsto {}^\perp Q)^\perp=\bigcap
  \{N^\perp \subseteq \Pi: N({}^\perp P) \subseteq {}^\perp Q\}$. Call
  $N_0={}^\perp(P \rightarrow Q)$, in accordance to the above equality
  in order to prove that $({}^\perp P \leadsto {}^\perp Q)^\perp
  \subseteq (P \rightarrow Q)=N_0^\perp$, all we have to show is that
  $N_0 ({}^\perp P) \subseteq {}^\perp Q$ or in other words that
  ${}^\perp(P \rightarrow Q)({}^\perp P) \subseteq {}^\perp Q$. Taking
  perpendiculars in the above inequality we show that our statement
  implies that $Q \subseteq (P \rightarrow Q) \diamond P$.  

Conversely, the inclusion $Q \subseteq (P \rightarrow Q) \diamond P$
implies that ${}^\perp(P \rightarrow Q)({}^\perp P) \subseteq {}^\perp
Q$, which in turn implies --by the definition of the conductor-- that
${}^\perp(P \rightarrow Q) \subseteq {}^\perp P \leadsto {}^\perp
Q$. Taking perpendiculars again in this inclusion we deduce that
$({}^\perp P \leadsto {}^\perp Q)^\perp \subseteq (P \rightarrow Q)$.   
\item Also, the two formulations of condition (3) are
  equivalent. Indeed, it is clear that ${}^\perp Q.R \subseteq P$ if
  and only if $(Q \rightarrow R)= \big({}^\perp({}^\perp
  Q.R)\big)^\perp\subseteq P$ and also it follows that ${}^\perp P
  {}^\perp Q \subseteq {}^\perp R$ can also be written as $R=
  ({}^\perp R)^\perp \subseteq ({}^\perp P {}^\perp Q)^\perp = P
  \diamond Q$. 

\item Assuming that the original formulation of rule (S1)
  holds, we want to prove (1), which is the assertion that for all
  $P,Q \in \mathcal{P}_\perp(\Pi)$, then: 
\[\{\pi \in \Pi: {}^\perp Q.\pi \subseteq P\} \subseteq ({}^\perp P
        {}^\perp Q)^\perp.\] 

In other words we want to show that if $\pi \in \Pi$ is such that
${}^\perp Q.\pi \subseteq P$ then, for all $s \perp P,\, t \perp Q$ we
have that $st \perp \pi$. It is clear that from the hypothesis
${}^\perp Q.\pi \subseteq P$ and $s \perp P,\, t \perp Q$, that $s
\perp t.\pi$ and in this case the original condition (S1) implies that
$st \perp \pi$.     

 \item Now we prove that condition (1) implies condition (3). Using
   the ``half adjunction property'' from the hypothesis of (3): $(Q
   \rightarrow R) \subseteq P$ we deduce that $R \subseteq P \circ Q$
   and using (1) we prove that $R \subseteq P \circ Q \subseteq P
   \diamond Q$.   
\item Next we prove that (3) implies (2). Consider the equality $(P
  \rightarrow Q)=(P \rightarrow Q)$ and using (3) deduce that $Q
  \subseteq (P \rightarrow Q) \diamond P$ that is exactly the
  statement of (2).  




\end{itemize}
\item Observe that both versions of condition (1) are equivalent. We
  prove first that our condition (1) implies the original condition
  (S2). Assume that $t \perp \pi$, we want to show that for all $s \in
  \Lambda$ we have that $K \perp (t.s.\pi)$. Call
  $P=({}^\perp\{\pi\})^\perp$ and $R=\{s\}^\perp$. From the assertion
  that $K \perp {}^\perp P.{}^\perp R.P$ as $t \in {}^\perp P,\, s \in
  {}^\perp R$ and $\pi \in P$ we conclude that $K \perp (t.s.\pi)$.  

Conversely, suppose the take the subset ${}^\perp P.{}^\perp R.P
\subseteq \Pi$ and we want to prove that for all $t \perp P$, $s \perp
R$ and $\pi \in P$, $K \perp t.s.\pi$. As $t \perp \pi$ from the
original condition (S2) we deduce that $K \perp (t.s.\pi)$ that is
exactly what we needed to prove.   The fact that condition (1) implies
condition (2) is a direct consequence of the axiom (S1) of a $\mathcal
{PAKS}$.  
\item The proof of this part uses the same methods than the previous one.
\item Axiom (S5) can be written as the assertion:
  $\operatorname{save}(P) \subseteq {}^\perp({}^\perp P.Q)={}^\perp(P
  \rightarrow Q)$ for all $P, Q$ and axiom (S4) can be written as the
  assertion: $\operatorname{cc} \in
  {}^\perp\Big({}^\perp\big(\operatorname{save}(P).P\big).P\Big)=
  {}^\perp\Big(\big(\operatorname{save}(P).P\big)   
  \rightarrow P\Big)$ for all $P$.  

Putting this together, we obtain that: \[\operatorname{cc} \in
{}^\perp\Big(\big(\operatorname{save}(P).P\big) \rightarrow P\Big) \subseteq
{}^\perp\Big(\big({}^\perp(P \rightarrow Q).P\big) \rightarrow P\Big)
\subseteq {}^\perp\Big(\big((P \rightarrow Q) \rightarrow P\big)
\rightarrow P\Big) \mbox{ for all $P$, $Q$}\in\mathcal{P}_{\perp}(\Pi).\] 
\end{list}
\end{proof}

\item In accordance with Theorem \ref{theo:adjunction} (\emph{half
  adjunction property}) we have that: if $Q \rightarrow R \subseteq
  P$, then $R \subseteq P \circ Q$. In search of a version of a
  converse to this result--i.e to obtain the other ``half'' of the
  adjunction, we introduce the so called ``\emph{$\operatorname E$
    operator}'' and the associated ``\emph{$\operatorname{S}\eta$
    rule}''. 
\begin{theo} In a $\mathcal{PAKS}$ if $t,s \in \Lambda$ we have
  that:\[ts \perp \pi \Rightarrow \operatorname{S}(\operatorname
  {K}(\operatorname{S}\operatorname {K}\operatorname {K}))t \perp
  s.\pi.\] 
\end{theo}
\begin{proof} The proof is performed in two steps.
\begin{enumerate}
\item If $t \perp \pi$, then $\operatorname S\operatorname{K}\operatorname{K} \perp t . \pi$. 
Indeed: \[t \perp \pi \Rightarrow \operatorname{K} \perp
t.(\operatorname{K}t).\pi \Rightarrow
(\operatorname{K}t)(\operatorname{K}t) \perp \pi \Rightarrow
\operatorname{S} \perp \operatorname{K}.\operatorname{K}.t.\pi
\Rightarrow \operatorname{S}\operatorname{K}\operatorname{K} \perp
t.\pi.\] The validity of the succesive implications come by respective
application of the following axioms
{\Large{\bf{\ref{item:axiomsperp}.}}} (S2),(S1),(S3), and (S1) in that
order.  
\item If $ts \perp \pi$, then $\operatorname{S}(\operatorname {K}(\operatorname{S}\operatorname {K}\operatorname {K}))t \perp s.\pi$.
The following chain of implications proves the result: 
\begin{eqnarray}
\nonumber ts\perp \pi \Rightarrow
\operatorname{S}\operatorname{K}\operatorname{K} \perp ts.\pi
\Rightarrow \operatorname{K}\perp
\operatorname{SKK}.s.ts.\pi\Rightarrow
\operatorname{K}(\operatorname{SKK})s(ts)\perp \pi \\\nonumber
\operatorname{K}(\operatorname{SKK})s(ts)\perp \pi \Rightarrow
\operatorname{S}\perp (\operatorname{K}(\operatorname{SKK})).t.s.\pi
\Rightarrow \operatorname{S}(\operatorname{K}(\operatorname{SKK}))t
\perp s.\pi. 
\end{eqnarray} The list of the axioms or results used at each
respective implication is: Part (1) above,
{\Large{\bf{\ref{item:axiomsperp}.}}} (S2), (S1), (S3) and (S1).   
\end{enumerate}
\end{proof}
\begin{defi} The special elements of $\Lambda$ considered above are
  abbreviated as follows: \[\operatorname{I}=\operatorname{SKK} \,;\,
  \operatorname{E}=\operatorname{S}(\operatorname{K}\operatorname{I})=\operatorname{S}(\operatorname{K}(\operatorname{SKK})).\] 
\end{defi}

Thus, the $\operatorname{S}\eta$ rule can be formulated as:
\[ts \perp \pi \Rightarrow \operatorname{E}t \perp s.\pi.\]

Next we present a set theoretical characterization of the $\operatorname{S}\eta$ rule that can be proved easily.
\begin{lema}\label{lema:setheoSeta}
\begin{enumerate}
\item A combinator $\widehat{\operatorname{E}}$ satisfies the $\operatorname{S}\eta$ rule --i.e. $ts \perp \pi \Rightarrow \widehat{\operatorname{E}}t \perp s.\pi$-- if and only if satisfies any of the the assertions that follow. 
\begin{equation}\label{eqn:setados} \text{If}\,\, P,Q \in \mathcal P_\perp(\Pi)\,\, \text{then}\,\, P \diamond Q \subseteq \{\pi \in \Pi: \widehat{\operatorname{E}}{}^\perp P \subseteq {}^\perp({}^\perp Q. \pi)\}=\{\pi \in \Pi: (\widehat{\operatorname{E}}{}^\perp P)^\perp \supseteq ({}^\perp Q. \pi)\}.
\end{equation}
\begin{equation}\label{eqn:setatres} \text{If}\,\, P,Q \in \mathcal P_\perp(\Pi)\,\, \text{then}\,\, \widehat{\operatorname{E}}{}^\perp P \subseteq {}^\perp\big({}^\perp Q.(P \diamond Q) \big). 
\end{equation} 
\begin{equation}\label{eqn:seta3.1} \text{If}\,\, R \subseteq (P \diamond Q),\, \text{with}\,\, P, Q, R \in P_\perp(\Pi)\,\, \text{then}\,\, \widehat{\operatorname{E}}{}^\perp P \subseteq {}^\perp({}^\perp Q. R).\footnote{In terms of subsets of $\mathcal P_\perp(\Lambda)$ it can be formulated as: $LL' \subseteq M \in \mathcal P_\perp(\Lambda)$, then $\widehat{\operatorname{E}}L \subseteq {}^\perp(L'.M^\perp)$.}
\end{equation}
\item If the combinator $\widehat{\operatorname{E}}$ satisfies the $\operatorname{S}\eta$ rule then, the assertions that follow --see the notations of Definition \ref{defi:maps}-- are valid.
\begin{equation}\label{eqn:setacuatro} \text{If}\,\, P,Q \in \mathcal P_\perp(\Pi)\,\, \text{then}\,\, \widehat{\operatorname{E}}({}^\perp({}^\perp P . Q)) \subseteq {}^\perp({}^\perp P . Q)\,\, \text{or equivalently}\,\, \widehat{\operatorname{E}}({}^\perp(P \rightarrow Q)) \subseteq {}^\perp(P \rightarrow Q).
\end{equation}
\begin{equation}\label{eqn:setauno}\big(t({}^\perp P)\big)^\perp \subseteq \{\pi \in \Pi: (\widehat{\operatorname{E}}t)^\perp \supseteq ( {}^\perp P. \pi)\} \subseteq ({}^\perp\{\pi \in \Pi: (\widehat{\operatorname{E}}t)^\perp \supseteq ( {}^\perp P. \pi)\})^{\perp}=(\widehat{\operatorname{E}}t)^\perp \circ P.
\end{equation}
\begin{equation}\label{eqn:seta3.5}\text{If}\,\, P,Q \in \mathcal P_\perp(\Pi)\,\, \text{then}\,\,(P \diamond Q) \subseteq (\widehat{\operatorname{E}}({}^\perp P))^\perp \circ Q.
\end{equation} 
\end{enumerate}
\end{lema} 
\begin{proof} 
\begin{enumerate}
\item  
\begin{itemize}
\item It is clear that the assertions \eqref{eqn:setados},\eqref{eqn:setatres},\eqref{eqn:seta3.1} are all equivalent. 

\item The inclusion \eqref{eqn:setados} is equivalent to the assertion: $\forall s,t,\pi$, $ts \perp \pi \Rightarrow \widehat{\operatorname{E}}t \perp s.\pi$. 

Assume that $ts \perp \pi$ and call $P=\{t\}^\perp$ and $Q=\{s\}^\perp$. Clearly $ts \in {}^\perp P {}^\perp Q$ and then $\pi \in P\diamond Q$ and in the situation that the inclusion \eqref{eqn:setados} is valid, we deduce that $\widehat{\operatorname{E}}({}^\perp P) \subseteq {}^\perp({}^\perp Q. \pi)$. As $t \in {}^\perp P$ and $s \in {}^\perp Q$, we obtain that $\widehat{\operatorname{E}}t \perp s.\pi$.  The converse can be proved by reversing the above argument.  

\end{itemize}
\item 
\begin{itemize}
\item For the proof of the fact the $\operatorname{S}\eta$\,\, rule implies the inclusion \eqref{eqn:setacuatro} we proceed as follows.  

Assume that $t \in {}^\perp({}^\perp P . Q)$, then $t \perp s.\pi$ for all $s \in {}^\perp P$ and $\pi \in Q$. In this situation we deduce that $ts \perp \pi$ and applying the $\operatorname{S}\eta$ rule we deduce that 
$\widehat{\operatorname{E}}t \perp s.\pi$. This means that $\widehat{\operatorname{E}} t \in {}^\perp({}^\perp P . Q)$. 
\item Next we show that the the $\operatorname{S}\eta$ rule implies the inclusion \eqref{eqn:setauno}.

Assume as hypothesis the validity of the $\operatorname{S}\eta$ rule. Take $\pi \in \big(t({}^\perp P)\big)^\perp$--i.e. assume that for all $s \perp P$, $ts \perp \pi$. Using the hypothesis we deduce that for all $s \perp P$ we have that $\widehat{\operatorname{E}}t \perp s.\pi$ and that means that ${}^\perp P.\pi \subseteq (\widehat{\operatorname{E}}t)^\perp$ and that implies that the inclusion \eqref{eqn:setauno} is valid. 
\item The validity of \eqref{eqn:seta3.5} is a consequence of the
  following chain of inclusions --the first one is just the inclusion
  \eqref{eqn:setados}--:
\begin{center}
$P \diamond Q \subseteq \{\pi \in \Pi:
(\widehat{\operatorname{E}}({}^\perp P))^\perp \supseteq ({}^\perp
Q. \pi)\} \subseteq \big({}^\perp\{\pi \in \Pi:
(\widehat{\operatorname{E}}({}^\perp P))^\perp \supseteq ({}^\perp
Q. \pi)\}\big)^\perp= (\widehat{\operatorname{E}}({}^\perp P))^\perp
\circ Q.$ 
\end{center}
\end{itemize}
\end{enumerate}
\end{proof}
\begin{coro}\label{coro:adjunctorpaks}For all $P \in \mathcal P_\perp(\Pi)$ and for $\operatorname{E}$ as before, we have that:
\begin{equation}\label{eqn:seta1.5} (E{}(^\perp P))^\perp \subseteq (\operatorname{EE})^\perp \circ P.
\end{equation}
\end{coro} 
\begin{proof}This assertion is a particular case of \eqref{eqn:setauno} when $t=\operatorname{E}$. 
\end{proof}
\begin{obse}
\begin{enumerate} 

\item Another consequence of the $\operatorname{S}\eta$ rule, that follows directly from the above results --see inclusion \eqref{eqn:seta3.5}, as well as Definition \ref{defi:maps}, \eqref{item:maps1}--is the following:
\begin{equation}\label{eqn:setadiamond} \text{If}\,\, P,Q \in \mathcal P_\perp(\Pi)\,\, \text{then}\,\, P \diamond Q \subseteq \Big({}^\perp\big({}^\perp Q \leadsto (\operatorname{E}({}^\perp P))^\perp\big)\Big)^\perp.
\end{equation}
\item Notice that we have proved that the operator $\operatorname{E}$ contracts subsets of $\Lambda$ of the form: ${}^\perp(P \rightarrow Q)$ for $P$ and $Q$ in the corresponding $\mathcal P_\perp(\Pi)$--see property \eqref{eqn:setacuatro}. It does not seem possible to prove that $\operatorname{E}$ contracts all subsets $L$ in $\mathcal P_\perp(\Lambda)$.
\item If we put together the above equation \eqref{eqn:seta3.5} and Lemma \ref{lema:axiomset}, ($\mathbb{S}1$), (1), we obtain:
\begin{equation}
P \circ Q \subseteq P \diamond Q \subseteq (\operatorname{E}({}^\perp P))^\perp \circ Q.
\end{equation} 
\end{enumerate}
\end{obse} 
The theorem that follows --that is of importance for future developments--is a partial converse to the half adjunction property of Theorem \ref{theo:adjunction}.
 
\begin{theo} \label{theo:adjunctionconverse} Let $P,Q,R \in \mathcal P_\perp(\Pi)$. If $P \circ Q \supseteq R$ then $\operatorname E {}^\perp P \subseteq {}^\perp(Q \rightarrow R)$. Equivalently, if $P \circ Q \supseteq R$ then $(\operatorname{E}{}^\perp P)^\perp \supseteq (Q \rightarrow R)$.   
\end{theo}
\begin{proof} 
As $R \subseteq P \circ Q \subseteq P \diamond Q= ({}^\perp P {}^\perp
Q)^\perp$ --see Lemma \ref{lema:axiomset} (S1)
\eqref{item:circdiamond}--, we have that ${}^\perp Q .R \subseteq
      {}^\perp Q. ({}^\perp P {}^\perp Q)^\perp$ and
      ${}^\perp({}^\perp Q .R) \supseteq {}^\perp({}^\perp
      Q. ({}^\perp P {}^\perp Q)^\perp)$. Using the inclusion \eqref{eqn:setatres} we deduce that $\operatorname{E}{}^\perp P \subseteq {}^\perp\big({}^\perp Q.({}^\perp P {}^\perp Q )^\perp \big) \subseteq {}^\perp({}^\perp Q .R)= {}^\perp(Q \rightarrow R)$ that is the inequality we wanted to prove.

Clearly the inequality $\operatorname E {}^\perp P \subseteq {}^\perp(Q \rightarrow R)$ is equivalent to $(\operatorname{E}{}^\perp P)^\perp \supseteq (Q \rightarrow R)$ --see Observation \ref{obse:repeatedly}--.  
\end{proof}
 
\item In order to summarize, we write down explicitly the adjunction
  properties valid in a general $\mathcal {PKAS}$. We also put them
  together --for future use-- with the conclusion of
  \eqref{eqn:seta1.5}.  
\begin{theo}\label{theo:main}
Assume that $P,Q,R \in \mathcal P_\perp(\Pi)$. 
\begin{eqnarray}
\label{eqn:direct} (Q \rightarrow R) \subseteq P &\Rightarrow& R \subseteq P \circ Q\\
\label{eqn:converse} R \subseteq P \circ Q &\Rightarrow& (Q \rightarrow R) \subseteq (\operatorname{E}{}^\perp P)^\perp \subseteq (\operatorname{EE})^\perp \circ P
\end{eqnarray}
\end{theo}

\item \label{item:quasiproofsaxioms} When we add to the $\mathcal {PAKS}$ a subset of terms called quasi proofs we obtain the concept of \emph{Abstract Krivine Structure} --$\mathcal {AKS}$. This last concept was introduced by J.L. Krivine and generalized by T. Streicher --see \cite{kn:kr2008} and \cite{kn:streicher} respectively--.
\begin{defi}\label{defi:aksmain}An \emph{Abstract Krivine Structure} is a decuple: \[(\Lambda,\Pi,\Perp,\operatorname{app},\operatorname{save},\operatorname{push}, \operatorname{K},\operatorname{S},\operatorname{cc},\operatorname{QP}),\] where the nonuple: \[(\Lambda,\Pi,\Perp,\operatorname{app},\operatorname{save},\operatorname{push}, \operatorname{K},\operatorname{S},\operatorname{cc}),\] is a $\mathcal {PAKS}$ and the subset $\operatorname{QP} \subseteq \Lambda$ whose elements are called \emph{quasi proofs} satisfies the following conditions:
\newcounter{rcounter}
\begin{list}{(S\roman{rcounter})}{\usecounter{rcounter}}
\item $\operatorname{K},\operatorname{S},\operatorname{cc} \in \operatorname{QP}$
\item $\operatorname{app}(\operatorname{QP},\operatorname{QP}) \subseteq \operatorname{QP}$. 
\end{list} 
\end{defi}

\begin{obse} \label{obse:EEQP}It is clear that if $\operatorname{QP}$ is as in Definition \ref{defi:aksmain}, then $\operatorname{E}$ as well as $\operatorname{EE}$ are elements of $\operatorname{QP}$.
\end{obse}
\item The abbreviations and notations introduced along this section, will be in force in this notes.
\section{Combinatory algebras and ordered combinatory algebras.}
\label{section:five}
\item We recall the definition of combinatory algebra --abbreviated as $\mathcal {CA}$--. 
\begin{defi} \label{defi:combalg} A combinatory algebra is a quadruple
  $(A,\circ,\operatorname{k},\operatorname{s})$ where $A$ is a set,
  $\operatorname{k},\operatorname{s} \in A$ is a pair of distinguished
  elements of $A$ and $\circ: A \times A \rightarrow A$ is an
  operation --written as $\circ(a,b)=ab$ and called the application of
  $A$. The data displayed above are subject to the axioms:
  $\operatorname{k}\!ab=a$ and $\operatorname{s}\!abc=ac(bc)$.   
\end{defi}

\begin{obse} The application taken above, is not
  necessarily associative, hence as it is customary we associate to
  the left: $abc=(ab)c$ , etc.   

The axioms introduced in Definition \ref{defi:combalg} mean:
\begin{enumerate}
\item $(\operatorname{k}\!a)b=a$; 
\item $((\operatorname{s}\!a)b)c=(ac)(bc)$.
\end{enumerate}
\end{obse}
\item We perform some manipulations in a $\mathcal {CA}$.
\begin{obse}
\begin{enumerate}
\item $\operatorname{skk}a=\operatorname{k}a(\operatorname{k}a)=a$, in
  other words the element $\operatorname{skk}$ behaves as the identity
  with respect to the operation in $A$.   

The first equality is a direct consecuence of the second axiom of a
combinatory algebra and the second equality follows directly from the
first --see Definition \ref{defi:combalg}--. 
\item $\operatorname{k}ab=a$, so that $\operatorname{k}$ works as the
  projection in the first coordinate. 
\item
  $\operatorname{k(skk)}ab=(\operatorname{k(skk)}a)b=\operatorname{skk}b=b$,
  so that $\operatorname{k(skk)}$ operates as the projection in the
  second coordinate.  

The first equality is just the law of parenthesis, the second is the
first axiom of a combinatory algebra and the third was just proved.  
\end{enumerate}
\end{obse}
\item Next we recall the manner in which $\lambda$--calculus can be
  reformulated in the above framework without performing substitutions
  when using reduction.  
\begin{defi} Assume that we have $\mathcal{V}$ a countable set of
  variables 
  that we denote as $x_1,x_2,\cdots$. Consider $\mathcal{U}\subseteq
  \mathcal{V}$ and define  
  $A[\mathcal{U}]$ as the smallest set containing $\mathcal{U}$,
  $\operatorname{k},\operatorname{s}$ and that is closed 
  under application. 
\end{defi}
Observe that each element of
  $A[\mathcal{V}]$ contains only a \emph{finite} number of
  variables and then \[A[\mathcal{V}]=\bigcup\{A[x_1, \dots,
    x_k]\ |\ k\in\mathbb{N}\}\]    
\begin{theo} There is a function $\lambda^*y:A[x_1,\cdots, x_k,y]
  \rightarrow A[x_1,\cdots, x_k]$ satisfying the following property: 
\[\forall t \in A[x_1,\cdots,x_k,y]\,\,,\,\,\forall u \in
A[x_1,\cdots,x_k]\quad \operatorname{then} \quad (\lambda^*y
(t))\circ u=t\{y:=u\}. \]  
\end{theo}
\begin{proof} We abbreviate $(\lambda^*y
(t))\circ u$ as $(\lambda^*y
(t)) u$. Denote $\lambda^*y(t)= \lambda^*y.t$. Define:
\begin{enumerate}
\item $\lambda^*y.t=\operatorname{k}t$ provided that $y$ does not
  appear in $t$. 
\item $\lambda^*y.y=\operatorname{skk}$.
\item $\lambda^*y.(tu)=\operatorname{s}(\lambda^*y.t)(\lambda^*y.u)$. 
\end{enumerate}
\end{proof}
\item Taking the above into account, one could define the standard
  Krivine abstract machine --abbreviated as $\mathcal {KAM}$-- in the
  following manner. 
\begin{defi}\label{defi:kam}
\begin{enumerate} 
\item The terms and stacks are:
\[\Lambda :
x\,|
\operatorname{K}|\operatorname{S}|\operatorname{cc}|\operatorname{k}_\pi|\,ts  
\quad;\quad \Pi : \alpha\,|\,t.\pi,\] 
and as before the elements of the set $\Lambda$ are called the terms and the elements of the set $\Pi$ are called the stacks. The element $\alpha$ is called a constant stack.

As before, the elements of $\Lambda \times \Pi$ are called processes and a generic process is denoted as $t \star \pi$. 
\item The reduction is defined by the following rules:
\newcounter{pcounter}
\begin{list}{(R\arabic{pcounter})}{\usecounter{pcounter}}
\item $ts\star \pi  \ssucc t \star s.\pi$;
\item $\operatorname{K} \star\, t.s.\pi \ssucc t \star \pi$;
\item $\operatorname{S} \star\, t.s.u .\pi \ssucc tu(su) \star \pi$;
\item $\operatorname{cc} \star\, t.\pi \ssucc t \star k_\pi.\pi$;
\item $\operatorname{k}_\pi \star\, t.\pi' \ssucc t \star \pi$.
\end{list}
\end{enumerate}
\end{defi} 
\begin{obse} It is worth noticing that the reduction rules introduced
  in Definition \ref{defi:kam} are equivalent to the assertion that
  $\Perp$ is closed by the antireduction determined by the rules
  written in  Definition \ref{defi:aks} item \eqref{item:aksaxioms}. 
\end{obse}
\begin{question} What are the differences between choosing as models one or the other of the following two contexts?
\[\mathcal{PAKS} \Leftrightarrow \mathcal{KAM}\]
\end{question}
\noindent
{\bf Partial answer:}
\begin{center}
\begin{tabular}{ | l | c| c |}
\hline
& $\mathcal{PAKS}$& $\mathcal{KAM}$\\ \hline
The processes & $\Lambda,\Pi$ are more general. & $\Lambda,\Pi$ are more standard. \\ \hline
The calculus & Can be more abstract. & Is more rigid. \\ \hline
\end{tabular}
\end{center}
For example, in a general $\mathcal{PAKS}$ the sets $\Lambda$ and
$\Pi$ could be the same. Moreover, in the situation of an abstract
$\mathcal{PAKS}$ the application can have properties that the standard
$\lambda$--calculus does not have, e.g. it can be commutative. 

\item As the definition of a $\mathcal{PAKS}$ does not involve an
  equality defined in advance, in order to relate this concept with
  the concept of a combinatory algebra, we need to relax the
  definitions and look at \emph{ordered combinatory algebras}
  \cite{kn:hofstra2006}.  
\begin{defi} \label{defi:oca}An ordered combinatory algebra
  --$\mathcal{OCA}$-- consists of the following:  
\begin{enumerate} 
\item A quintuple \[(A,\circ,\leq,\operatorname{k},\operatorname{s}),\] where:
\begin{enumerate}
\item $A$ is a set.
\item $\circ: A \times A \rightarrow A$ is a function $(a,b)\mapsto
  \circ(a,b)=a\circ b$ --the function $\circ$ is called the
  \emph{application} and concerning this application we always associate
  to the left--. 
\item The relation $\leq$ is a partial orden in $A$\footnote{Recall
  that a partial order in $A$ is a relation $\leq\,\, \subseteq A
  \times A$ with the following properties:  (1) Reflexivity: $a \leq
  a$; (2) Antisymmetry: $a \leq b$, $b \leq a$ implies, $a=b$; (3)
  Transitivity: $a \leq b$ and $b \leq c$, imply $a \leq c$. A partial
  order that do not necessarily satisfies (2), is called a
  preorder.}. 
\item $\operatorname{k}$ and $\operatorname{s}$ are a pair of
  distinguished elements of $A$. 
\end{enumerate}
\item The above ingredients are subject to the following axioms.
\begin{enumerate}
\item The map $\circ:A \times A \rightarrow A$ is monotone with respect to the cartesian product order in $A \times A$ --i.e. if $a \leq a'$ and $b \leq b'$, then $ab \leq a'b'$--. 
\item \label{item:ks} The distinguished elements satisfy:
\begin{enumerate}
\item $\operatorname{k}\!ab \leq a$; 
\item $\operatorname{s}\!abc \leq ac(bc)$.
\end{enumerate}
\end{enumerate}
\item  \label{item:application}We say that the $\mathcal {OCA}$ is
  equipped with an implication, if there is a binary operation
  --called implication-- $\rightarrow: A \times A \rightarrow A$ with
  the following properties: 
\begin{enumerate}
\item \label{item:halfadj}(\emph{Half Adjunction property}.) For all $a,b,c \in A$, if $a \leq (b \rightarrow c)$ then $ab \leq c$. 
\item \label{item:monotony} The map $\rightarrow: A \times A \rightarrow A$ is monotone in the second variable and antimonotone in the first.
\end{enumerate}
\item (\emph{Adjunction property}.) We say that the $\mathcal {OCA}$
  with implication has the \emph{complete adjunction property} or
  simply the $\emph{adjunction property}$ if there is a distinguished
  element $\operatorname {e} \in A$, with the property that for all
  $a,b,c \in A$, if $ab \leq c$ then $\operatorname{e}a \leq (b
  \rightarrow c)$. The element $\operatorname{e}$ is called an
  \emph{adjunctor}.  
\item We say that the $\mathcal {OCA}$ $A$ is \emph{classic}, if there is an element $\operatorname{c}$ with the property that $\operatorname{c} \leq (((a\rightarrow b)\rightarrow a)\rightarrow a)$. 
\item A subset $B \subset A$ is a sub--$\mathcal {OCA}$ if:
\begin{enumerate}
\item $\circ(B \times B) \subseteq B$.
\item $\operatorname{k},\operatorname{s} \in B$.
\item If the original $\mathcal {OCA}$ has an implication   
$\rightarrow$, we ask $B$ to satisfy that $\rightarrow(B \times B)
  \subseteq B$.  
\item In the situation that $A$ has an adjunctor $\operatorname{e} \in A$, we assume that $\operatorname{e} \in B$.
\end{enumerate}
 \end{enumerate}
\end{defi}
\begin{obse} 
\begin{enumerate}
\item It is clear that if $B \subseteq A$ is a sub--$\mathcal {OCA}$,
  then $(B,\circ|_{B \times B}, \leq|_{B\times B},
  \operatorname{k},\operatorname{s})$ is 
  also an $\mathcal {OCA}$. Moreover, if $A$ has an implication 
  $\rightarrow$, then the restriction $\rightarrow|_{B \times B}$ is
  an implication for $(B,\circ|_{B \times
    B}, \leq|_{B\times B},
  \operatorname{k},\operatorname{s})$. Similarly, if 
  $\operatorname{e}$ is an adjunctor for $A$ that belongs to $B$, it is
  also an adjunctor for $B$.  
\item The property above --Definition \ref{defi:oca},
  \eqref{item:halfadj} is called ``half adjunction property'', because
  of the following. If we fix $x \in A$, the morphisms $R_x: y \mapsto
  (x \rightarrow y) :A \rightarrow A$ and $L_x:y \mapsto (x\circ y): A
  \rightarrow A$, satisfy the property that $a \leq R_b(c)$ implies
  that $L_b(a) \leq c$. If we view the preorder set $(A, \leq)$ as a
  category and the maps $L_x,R_x$ as functors, the equivalence $a \leq
  R_b(c)$ if and only if $L_b(a) \leq c$ can be stated as: for all $x
  \in A$ the functor $R_x$ is the right adjoint of $L_x$.  
\item In case that the original $\mathcal {OCA}$ has an adjunctor, we
  have the following situation: for all $a,b,c \in A$:    
\[a \leq R_b(c) \Rightarrow L_a(b) \leq c \Rightarrow
\operatorname{e}a \leq R_b(c).\] 
\end{enumerate}   
\end{obse}
\begin{defi} \label{defi:suboca}Assume that in (the $\mathcal {OCA}$)
  $A$, we have a subset $X \subseteq A$. Define the sub--$\mathcal{OCA}
  , \langle X \rangle=\bigcap \{B \subseteq A:  X \subseteq B,
  \text{$B$ sub $\mathcal{OCA}$ of}\,\, A\}$. 

In the case $A$ has an adjunctor, we assume that $\mathcal{OCA}$s we take in the intersection always contain $\operatorname{e}$. This is in order to guarantee that $\langle X \rangle$ has an adjunctor.   
\end{defi}
\begin{obse} It is important to remark the following difference. In
  combinatory algebras the concept or reduction is not present, only
  the concept of computation. In the present context, the symbol
  $\leq$ should be interpreted as ``reduces to''.   
\end{obse}
\item We perform some computations in the $\mathcal{OCA}$.
\begin{lema}\label{lema:ocaoperations} If $A$ is an $\mathcal {OCA}$, the following properties are valid.
\begin{enumerate} 
 \item If for $b \in A$ we call $\operatorname{i}_b= \operatorname
   {sk}b$ we have that $\operatorname{i}_ba \leq a$ for all $a \in
   A$. In particular the same is valid for
   $\operatorname{i}=\operatorname{i_k}$ 
\item $\operatorname{ki}a=\operatorname{k(skk)}a \leq \operatorname{skk}$ and $\operatorname{ki}ab=\operatorname{k(skk)}ab \leq b$.
\item Call $\operatorname{e}_0=\operatorname{s(ki)}$, then $\operatorname{e}_0ab \leq ab$.  
\item In particular $\operatorname{e}_0\operatorname{e}_0 a \leq \operatorname{e}_0 a$. 
\end{enumerate} 
\begin{proof}
\begin{enumerate} 
 \item We have that: $(((\operatorname{sk})b)a) \leq
   (\operatorname{k}a)(ba) \leq a$, using the conditions appearing in
   Definition \ref{defi:oca}, \eqref{item:ks}.  
\item We prove the second inequality, the first is similar: $\operatorname{k(skk)}ab=(\operatorname{k(skk)}a)b \leq \operatorname{skk}b \leq b$.
\item  $\operatorname{e}_0ab= \operatorname{s(ki)}ab \leq (\operatorname{ki}b)(ab)\leq \operatorname{i}(ab)\leq ab$. 
\item The inequality $\operatorname{e}_0\operatorname{e}_0 a \leq \operatorname{e}_0a$, follows directly from the previous result.  
\end{enumerate}
\end{proof}
\end{lema}
\item Let $A$ be an $\mathcal{OCA}$, we introduce the concept of filter in $A$. 
\begin{defi}\label{defi:filter}
A subset $\Phi \subseteq A$ is said to be a filter if:
\newcounter{mcounter}
\begin{list}{(F\arabic{mcounter})}{\usecounter{mcounter}}

\item The subset $\Phi$ is closed under application.
\item $\operatorname{k},\operatorname{s} \in \Phi$. 
\item If $A$ has an adjunctor $\operatorname{e}$, then $\operatorname{e} \in \Phi$. 
\item If $A$ is classic, we assume that $c \in \Phi$. 
\end{list}
\end{defi}
\begin{obse} It is clear that given $A$ and $\Phi$ as above, if we
  restrict to the filter the application and the order, then $\Phi$
  becomes a sub $\mathcal{OCA}$ of $A$.  
\end{obse}
\item In what follows, we will program directly in the
  $\mathcal{OCA}$, using the standard codifications in the
  combinatory algebras. 

\begin{defi} Let $A$ be an $\mathcal {OCA}$ and take a countable set
  of \emph{variables}: $\mathcal{V}=\{x_1,x_2,\cdots\}$. Consider
  $A(\mathcal{V})$ --called \emph{the set of terms in $A$}-- that is
  the set of formal 
  expressions given by the following grammar:
  \[
  p_1, p_2 ::= a \quad |\quad x\quad |\quad p_1 p_2
  \]
  where $a\in A$ and $x\in\mathcal{V}$. We denote as $A(x_1,\dots,
    x_k)$ the set of terms in $A$ containing only the variables $x_1,\cdots,x_k$. The term $p_1p_2$ is called the application of $p_1$ and $p_2$.   
\end{defi}
We can endow canonically a quotient of $A(\mathcal{V})$ with an $\mathcal{OCA}$ structure.  
\begin{obse} Consider the  --minimal-- partial preorder $R$ on $A(\mathcal{V})$ defined by the following statements: 
  \begin{enumerate}
    \item For $a, b\in A$ and if $a\leq b$, then $a \ R \ b$.
    \item For all $a, b\in A$: $a b \ R \ a\circ b$ and
      $a\circ b \ R \ a b$. 
    \item If $p_1, p_2, q_1, q_2\in A(\mathcal{V})$ are such that $p_1
      \ R \ p_2$ and $q_1 \ R \ q_2$ then $p_1q_1
      \ R \ p_2 q_2$. 
    \item If $p_1, p_2\in A(\mathcal{V})$ then $\operatorname{k} p_1 p_2 \ R
      \ p_1$. 
    \item If $p_1, p_2, p_3\in A(\mathcal{V})$ then $\operatorname{s} p_1 p_2 p_3
      \ R \ p_1 p_3 (p_2 p_3)$. 
  \end{enumerate}

Notice that this minimal preorder exists because we can take the intersection of the non empty family of preorders that satisfy the above conditions and the family is not empty because it always contains the trivial relation $A(\mathcal{V}) \times A(\mathcal{V})$.
   
  Define an equivalence relation $\equiv_R$ on $A(\mathcal{V})$ as: $p
    \equiv_R q$ iff $p \ R \ q$ and $q \ R \ p$. Thus, the order $R$ can be factored to the
    quotient $A[\mathcal{V}]:=A(\mathcal{V})/\equiv_R$ endowing it with a partial order. This quotient is called the set of polynomials in $A$. Observe that $ab \equiv_R a\circ b$ for
    all $a, b\in A$.  
    
    In order to simplify notations, we will use the same symbol $p$ to denote a polynomial (an element of the quotient) as well as for a term which belongs to the equivalence class of $p$.  

    Observe that in accordance with (3), if $p_1, p'_1, p_2, p'_2$ are terms such that
    $p_1\equiv_R p'_1$ and $p_2\equiv_R p'_2$ then $p_1 p_2 \equiv_R
    p'_1 p'_2$. Thus the application of terms induces a corresponding ``application'' of polynomials that we denote as $p_1 \star p_2$. 

    Then, by definition, $(A[\mathcal{V}], R, \star)$ is an $\mathcal{OCA}$ and $(A,
    \leq, \circ)$ is a sub-$\mathcal{OCA}$ of $(A[\mathcal{V}], R, \star)$.

    Abusing slightly the notations and when there are not
    possibilities of confusion, we denote the relation $R$ as $\leq$
    and the operation~$\star$ as~$\circ$ or as the concatenation of
    the factors. Also we call the elements of $A[\mathcal V]$
    \emph{terms} instead of \emph{polynomials}. We say that
    $(A[\mathcal{V}], \leq, \circ)$ is an extension of $(A, \leq,
    \circ)$.    
\end{obse}
\begin{theo} \label{theo:calculusinA} For any finite set of variables
  $\{x_1,\cdots,x_k,y\} $, 
  there is a function $\lambda^*y:A[x_1,\cdots,x_k,y] \rightarrow
  A[x_1,\cdots, x_k]$ satisfying the following property: 
  \begin{equation}\label{eqn:basic} \mbox{If }t \in
    A[x_1,\cdots,x_k,y]\,\,,\,\,\mbox{and } u \in 
  A[x_1,\cdots,x_k] \quad \operatorname{then} \quad (\lambda^*y
  (t))\circ u \leq t\{y:=u\}.
\end{equation}   
  Moreover if $X \subseteq A$ is an arbitrary subset and $t$ is a term
  with all its coefficients in 
  $X$, then $\lambda^*y(t)$ is a term with all its coefficients in
  $\langle X \rangle$. In particular if all the coefficients of $t$ are in the
  filter $\Phi$, then $\lambda^*y(t)$ is a polynomial with all the
  coefficients in $\Phi$. 
\end{theo}
\begin{proof} We give the following recursive definition for
  $\lambda^*y$: 
  \begin{itemize}
  \item If $y$ does not appears in $t$, then $\lambda^*y(t):= \operatorname{k}t$
  \item $\lambda^*y(y):= \operatorname{skk}$
  \item If $p, q$ are polynomials in $A[x_1, \dots, x_k, y]$, then
    $\lambda^*y(p q):= \operatorname{s}(\lambda^*y(p)) (\lambda^*y(q))$
  \end{itemize}
  Next we show that this function satisfies the requirements. Let us
  consider $u\in A[x_1, \dots, x_k]$. If $x\neq y$ then
  $(\lambda^*y(x)) u=\operatorname{k} x u\leq
  x=x\{y:=u\}$. $(\lambda^*y(y)) u=\operatorname{skk}u 
  \leq \operatorname{k}u(\operatorname{k}u)\leq u$. Suppose now that $p, q$ are such that
  $(\lambda^*y(p)), (\lambda^*y(p))$ when applied to $u$ satisfy  the
  inequalities \eqref{eqn:basic}. Then,
  $(\lambda^*y(pq))u=\operatorname{s}(\lambda^*y(p)) (\lambda^*y(q)) u\leq 
  (\lambda^*y(p)) u ((\lambda^*y(q))u)\leq p\{y:=u\}
  q\{y:=u\}=pq\{y:=u\}$.       

  Observe that, since $\langle X\rangle$ contains $\operatorname{k, s}$ and is closed
  under applications, then the condition on the coefficients of
  $\lambda^*y(t)$ follows by induction.  
\end{proof}
\begin{obse}\begin{enumerate}
\item Sometimes we write $\lambda^*y(t)=\lambda^*y.t$
\item Since application is monotone in both arguments, the proof of Theorem \ref{theo:calculusinA} can be interpreted as a method to translate lambda terms into elements of $A[\mathcal{V}]$ in such a way that $\leq$ 
reflects $\beta$-reduction.
\item
Moreover, the condition on the coefficients guarantees that lambda terms are translated as polinomials with coefficients on $\langle
\emptyset\rangle$ which is included into any filter $\Phi$ (it is in fact
the minimal filter of $A$). In
particular, a closed lambda term is translated as a constant
polynomial with coefficients on $\Phi$, which is identified with an
element of $\Phi$.
\end{enumerate}  
\end{obse}
\begin{theo} \label{theo:compoca}If $A$ is an $\mathcal {OCA}$, then:
\begin{enumerate}
\item \label{item:compoca1}There are elements $p,p_1,p_2 \in \Phi$ with the following properties: 
\begin{equation}\label{eqn:pairing}\forall a,b \in A\,,\, p_1(p ab) \leq a \,;\, p_2(pab) \leq b.
\end{equation}
It is customary to call $pab=a \wedge b$ and in that case the properties above --Equation \eqref{eqn:pairing}-- read: 
\begin{equation}\label{eqn:projections}\forall a,b \in A\,,\, p_1(a\wedge b) \leq a \,;\, p_2(a \wedge b) \leq b.
\end{equation}
\item There is an $f \in \Phi$ such that for all $a,b \in A$ we have that
\begin{equation} \label{eqn:commutation} (fa)b \leq ba.
\end{equation}
\item \label{item:compoca2}There are functions $D,E,F,G: A \rightarrow A$ and $M:A \times A \rightarrow A$ such that for all $a,b,c \in A$, then:
\begin{eqnarray}
\label{eqn:cero} ((D(a)c)b) \leq c(ab) &,& ((E(a)b)c) \leq c(ab) \\
\label{eqn:primera}(F(c)a)b &\leq& c (ab)\\
\label{eqn:segunda}G(c)(pab) &\leq& (ca)b \\
\label{eqn:cuarta}M(c,b)a &\leq& (ca)b.
\end{eqnarray} 
Moreover: $D(\Phi) \subseteq \Phi, E(\Phi) \subseteq \Phi\,,\,F(\Phi) \subseteq \Phi\,,\,G(\Phi) \subseteq \Phi$ and $M(\Phi,\Phi) \subseteq \Phi$. 

\end{enumerate} 
\end{theo}
\begin{proof} 
\begin{enumerate}
\item Define $p=\lambda^*x_1\lambda^*x_2\lambda^*x_3 x_3x_1x_2\,,\,p_1=\lambda^*x_1x_1\operatorname{k}\,,\,p_2=\lambda^*x_1 x_1 \operatorname{k}'$; where $\operatorname{k}'=\lambda^*x_1.\lambda^*x_2.x_2$.
\item Consider $f=\lambda^*x_1\lambda^*x_2x_2x_1$. In this situation it is clear that $(fa)b \leq ba$. 
\item Define $D(a)=\lambda^*x\lambda^*y x(ay)$, $E(a)=\lambda^*x\lambda^*y y(ax)$, $F(c)=\lambda^*x\lambda^*y c(xy)$, $G(c)=\lambda^*x (c(p_1x))(p_2x)$ and $M(c,b)=\lambda^*x. (cx)b$.

\end{enumerate}
\end{proof}

For later use we prove some properties of the \emph{meet} or \emph{wedge} operator. 
\begin{lema}\label{lema:meetoper} Assume that $A$ is an $\mathcal{OCA}$ as above --Definition \ref{defi:oca}--. 
\begin{enumerate}
\item \label{item:meetcero} The operator $\wedge: A \times A \rightarrow A$ is monotone in both variables, i.e. $a \leq a', b \leq b'$ implies that $a \wedge b \leq a'\wedge b'$
\item \label{item:meetuno} There is a map $R:A \rightarrow A$ with the property that for all $a,b,c \in A$ we have that $R(c)(a \wedge b) \leq a \wedge (cb)$. Moreover
$R(\Phi)\subseteq \Phi$. 
\end{enumerate}  
\end{lema}
\begin{proof}
\begin{enumerate}
\item This part follows directly from the fact that the application in
  $A$ is monotone in both variables.  
\item The following chain of inequalities yields the result.
\[a \wedge (cb) = (pa)(cb) \geq (D(c)(pa))b \geq \Big(\big(F(D(c))p\big)a\Big)b \geq G(F(D(c))p)(pab) \geq R(c)(a\wedge b).\]
Where we denoted $G(F(D(c))p)=R(c)$.  The justification of the chain of inequalities is the following going from left to right: \eqref{eqn:cero}, \eqref{eqn:primera}, \eqref{eqn:segunda}.
\end{enumerate}
\end{proof}
We need some consequences of Theorem \ref{theo:compoca}, that we record here for later use.
\begin{coro} \label{coro:globaladj} \begin{enumerate}
\item There is a function $H:A \times A \rightarrow A$ with the property that for all $a,b,c,m,n \in A$, we have that:
\begin{equation}\label{eqn:globaladj}
m((na)b) \leq c \Rightarrow H(m,n)a \leq (b \rightarrow c).
\end{equation}
Moreover, the function $H$ satisfies that $H(\Phi,\Phi) \subseteq \Phi$. 
\item \label{item:preparticular} In the previous notations, for any $a,b,c \in A$ we have that 
\[ (F(\operatorname{e})F(c))(a \rightarrow b) \leq (a \rightarrow (cb)).\]
\item \label{item:particular} In the previous notations, for any $a,b \in A$ we have that 
\[(F(\operatorname {e}) f) a \leq b \rightarrow ba.\]
In particular \[(F(\operatorname {e}) f) a \leq \operatorname{i} \rightarrow a.\]
\item \label{item:partconv} For all $a,b \in A$ as for $f \in \Phi$ as above, we have that: \[(fb)(b \rightarrow a) \leq a.\]
In particular \[(f\operatorname{i})(\operatorname{i} \rightarrow a) \leq a.\]
\end{enumerate}
\end{coro}
\begin{proof} The proof follows from previous constructions. 
\begin{enumerate}
\item By applying a few times inequality \eqref{eqn:primera} we have that:
\[((F^2(m)n)a)b \leq (F(m)(na))b \leq m((na)b) \leq c.\]
By the basic property of the adjunctor we deduce that:
$\operatorname{e}((F^2(m)n)a) \leq (b \rightarrow c)$. Using again the
inequality \eqref{eqn:primera} we obtain that:
$\Big(F(\operatorname{e})\big(F^2(m)n\big)\Big)a \leq
\operatorname{e}\big((F^2(m)n)a\big) \leq (b \rightarrow c)$. Then,
this part is proved by taking
$H(m,n)=F(\operatorname{e})\big(F^2(m)n\big)$. 
\item Starting from $(a \rightarrow b) \leq (a \rightarrow b)$ we
  deduce that $(a\rightarrow b)a \leq b$ and then $c((a \rightarrow
  b)a) \leq cb$. By using inequality \eqref{eqn:primera} we deduce
  that $(F(c)(a \rightarrow b))a \leq cb$ and by the property of the
  adjunctor we deduce that $e(F(c)(a\rightarrow b))\leq (a \rightarrow
  cb)$. Then, the proof can be finished using again inequality
  \eqref{eqn:primera}.  
\item Starting from --see \eqref{eqn:commutation}-- $(fa)b \leq ba$ we
  deduce that $\operatorname{e}(fa) \leq b \rightarrow ba$. Using
  \eqref{eqn:primera}, we conclude that $(F(\operatorname{e})f)a \leq
  b \rightarrow ba$. The rest of the assertion is guaranteed if we
  take $b=\operatorname{i}$.  
\item By the definition of $f$ we have that $f b (b\to a)\leq (b\to a)b$ which
  is less or equal than $a$. 
\end{enumerate}
\end{proof} 
\begin{obse}\label{obse:converseadj}
\begin{enumerate}
\item \label{item:converseadj} Concerning the converse of the above Corollary \ref{coro:globaladj}, one has the following easy result that is a direct consequence of the inequality \eqref{eqn:segunda} defining the function $G$.

For the function $G:A \rightarrow A$ we have that for all $m,a,b,c \in A$:
\[ ma \leq (b \rightarrow c) \Rightarrow G(m)(a \wedge b)=G(m)((pa)b) \leq c. \]

Indeed, the basic half adjunction property guarantees that $ma \leq (b \rightarrow c) \Rightarrow (ma)b \leq c$. The rest follows from the definition of $G$.
Moreover, the function $G$ is such that $G(\Phi) \subset \Phi$ and the element $p \in \Phi$.
\item It is interesting to consider the following. The  proof of Corollary \ref{coro:globaladj} uses strongly the property of the existence of the adjunctor $\operatorname{e}$ in the $\mathcal {OCA}$. Here we show a converse, i.e. if the result \eqref{eqn:globaladj} is valid, the existence of the adjunctor can be deduced. 

Indeed, if we assume that $ab \leq c$, applying twice the fact that  $\operatorname{i}d \leq d$ for all $d \in A$, we conclude that $\operatorname{i}((\operatorname{i}a)b) \leq (\operatorname{i}a)b \leq ab\leq c$. Hence using the result of Corollary \ref{coro:globaladj}, we deduce that $H(\operatorname{i},\operatorname{i})a \leq (b \rightarrow c)$. Hence, the element 
 $\operatorname{e} = H(\operatorname{i},\operatorname{i}) \in \Phi$, works as an adjunctor.
\end{enumerate}   
\end{obse}
\item In what follows we construct in an $\mathcal {OCA}$ with a filter $\Phi$ a new partial order (not necessarily antisymmetric) that will be used to construct a tripos from the $\mathcal {OCA}$. 
\begin{defi}\label{defi:squareord} Assume that the quintuple $(A,\circ,\leq,\operatorname{k},\operatorname{s},\Phi)$ is an $\mathcal {OCA}$ with a filter. We define the relation $\sqsubseteq_\Phi$ in $A$ as follows: 
\[a \sqsubseteq_\Phi b, \text{if and only if}\,\, \exists f \in \Phi: f\circ a \leq b.\]
\end{defi} 

Usually we omit the subscript $\Phi$ in the notation of the relation $\sqsubseteq_\Phi$, and as usual omit the symbol $\circ$ when dealing with the application in $A$ that is written $a \circ b=ab$. 

\begin{lema} \label{lema:propsq}In the context of Definition \ref{defi:squareord}, we have the following properties of $\sqsubseteq$.
\begin{enumerate}
\item \label{item:propsq1}  The relation $\sqsubseteq$ is a partial order in $A$ --not necessarilty antisymmetric--. 
\item \label{item:propsq2}The partial order $\leq$ is stronger than
  $\sqsubseteq$ (i.e. if $a \leq b$, then $a \sqsubseteq b$). 
\item The order $\sqsubseteq$ has the following compatibility relation
  with the application on $A$: for all $a,a',b,b'\in A$ we have that 
\[\sqleq{a}{b}\quad \text{and} \quad \sqleq{a'}{b'} \Rightarrow
\sqleq{a \wedge a'}{bb'}.\] 
\item  If $f \sqsubseteq (a \rightarrow b)$ with $f \in \Phi$, then $a \sqsubseteq b$.
\item If $A$ has an adjunctor, then for all $a,b \in A$, $a \sqsubseteq b$ if and only if there is an element $f \in \Phi$ such that $f \leq a \rightarrow b$. 
\end{enumerate}
\end{lema}
\begin{proof}
\begin{enumerate}
\item 
\begin{enumerate}
 \item $a \sqsubseteq a$ is a consequence of the fact that $\operatorname{i} a \leq a$ --see Lemma \ref{lema:ocaoperations}. Observe that being $\Phi$ closed under the operation of $A$, the element $\operatorname{i} \in \Phi$.  
\item If $a \sqsubseteq b$ and $b \sqsubseteq c$, then $a \sqsubseteq c$.
Indeed, by definition we can find $g,f \in \Phi$ such that:
\[ga \leq b \quad,\quad fb \leq c,\]
and using the monotony of the operation of $A$ we deduce that $f(ga) \leq fb \leq c$. Using Theorem \ref{theo:compoca},\eqref{item:compoca2} we deduce that there is an $h \in \Phi$ such that $ha \leq f(ga) \leq c$, that is our conclusion. 
\end{enumerate}
\item Suppose that  $a \leq b$, then $\operatorname{i}a \leq a \leq b$ so that $a \sqsubseteq b$.
\item By hypotesis, there exist $f,f'\in \Phi$ with the property that: $fa \leq b$ and $f'a'\leq b'$. Call $a_0=a \wedge a'$ and recall that $p_1a_0 \leq a$ and $p_2a_0 \leq a'$ as in Theorem \ref{theo:compoca} \eqref{item:compoca1}. Then $(F(f)p_1)a_0 \leq f(p_1a_0)\leq fa$ and $(F(f')p_2)a_0 \leq f'(p_2a_0)\leq f'a'$ --see \eqref{eqn:primera}. If we abbreviate: $g_1=F(f)p_1, g_2=F(f)p_2$ we deduce that $(g_1a_0)(g_2a_0)\leq bb'$. Using the basic property of $\operatorname{s}$ we obtain that $\operatorname{s}g_1g_2a_0 \leq (g_1a_0)(g_2a_0) \leq bb'$ and reducing again using the inequality \eqref{eqn:primera} we deduce that for some $h \in \Phi$ --depending only on $s,g_1,g_2$--, it is verified that $ha_0 \leq bb'$. This is our conclusion.   
\item If $a \sqsubseteq b$, then for some $f \in \Phi$ we have that $fa\leq b$, then $ef \leq a \rightarrow b$. Conversely, if $f \leq a \rightarrow b$ for $f \in \Phi$, then $fa \leq b$ and $a \sqsubseteq b$.  
\item  If $f \sqsubseteq (a \rightarrow b)$, then there is a $g \in \Phi$ such that $gf \leq (a \rightarrow b)$ and then $(gf)a \leq b$ and then $a \sqsubseteq b$. 
\end{enumerate}
\end{proof}

The theorem that follows, guarantees the complete adjunction property in an $\mathcal {OCA}$ with adjunctor, with respect to the order $\sqsubseteq$ , the ``meet'' operation and the arrow. It will be important for the categorification of the structures.
 
\begin{theo}\label{theo:propsmain} If the original $\mathcal {OCA}$ has an adjunctor, then the partial order $\sqsubseteq$ satisfies the following \emph{``adjunction property''} with respect to the operations $\wedge,\rightarrow$\footnote{Reading the proof the reader may verify that the existence of the adjunctor is not necessary to prove the assertion: $a\sqsubseteq (b \rightarrow c) \Rightarrow a \wedge b \sqsubseteq c$.}\emph{:}  
\begin{equation*}  a \wedge b \sqsubseteq c \ssi a\sqsubseteq (b \rightarrow c).
\end{equation*}
\end{theo}
\begin{proof} Assume that $a \sqsubseteq (b \rightarrow c)$, then for some $f \in \Phi$, $fa \leq (b \rightarrow c)$ and then $(fa)b \leq c$. From the inequality \eqref{eqn:segunda}, we deduce that $G(f)(pab)\leq (fa)b \leq c$ and then that $a \wedge b \sqsubseteq c$. 

Conversely, if we assume that $a \wedge b \sqsubseteq c$, then $f((pa)b) \leq c$ for some $f \in \Phi$. Then applying the inequality \eqref{eqn:primera}, we deduce that $(F(f)(pa))b \leq f((pa)b) \leq c$. Applying again the same inequality to the first factor we obtain that  $((F^2(f)p)a)b \leq (F(f)(pa))b \leq c$ and then we deduce that: $\operatorname{e}((F^2(f)p)a) \leq (b \rightarrow c)$. Using once again the inequality \eqref{eqn:primera} we obtain that $(F(\operatorname{e})(F^2(f)p))a \leq \operatorname{e}((F^2(f)p)a) \leq (b \rightarrow c)$, then as $\operatorname{e},F(\operatorname{e}),f,F(f),F^2(f),p \in \Phi$, we conclude that $a \sqsubseteq (b \rightarrow c)$. 
\end{proof}
\section{Construction of an $\mathcal {OCA}$ from a $\mathcal {PAKS}$.}
\label{section:six}
\label{section:pakstooca}
\item In this section we show how to perform a natural construction of an $\mathcal {OCA}$ from a $\mathcal {PAKS}$.
\begin{defi}\label{defi:paksoca}
Assume we have a $\mathcal {PAKS}$
\[(\Lambda,\Pi,\Perp,\operatorname{app},\operatorname{save},\operatorname{push}, \operatorname{K},\operatorname{S},\operatorname{cc}),\]
and define a set $A$, an order, an application, an
implication, the combinators $\operatorname{k, s}$, and an 
adjunctor $\operatorname{e}$ in the following manner: 
\begin{enumerate}
\item $A=\mathcal P_\perp(\Pi)$;
\item For a pair of elements $a,b \in A$ we say that $a \leq b$ iff $a \supseteq b$. 
\item For a pair of elements $a,b \in A$ we define $a \circ b$ as in Definition \ref{defi:maps}, \eqref{item:maps1}. In other words: \[a\circ b=({}^{\perp}\{\pi \in \Pi: \forall t \in {}^\perp a , \forall s \in {}^\perp b \quad t \perp s.\pi\})^{\perp}=({}^{\perp}\{\pi \in \Pi: {}^\perp a \subseteq ({}^\perp b .\,\pi)^\perp\})^{\perp}=({}^\perp \{\pi \in \Pi: a \supseteq ({}^\perp b .\,\pi)\})^{\perp}.\]
\item For a pair of elements $a,b \in A$ we define $a \rightarrow b$ as in Definition \ref{defi:maps}, \eqref{item:maps2}. In other words: \[a \rightarrow b= ({}^\perp\operatorname{push}({}^\perp a,b))^{\perp}=({}^\perp(a^{\perp} \cdot b))^{\perp}.\]
\item We define the following elements of $A$:
\[\operatorname{k}=\{\pi \in \Pi: \operatorname{K} \perp \pi\}=\{\operatorname{K}\}^\perp\,,\,\operatorname{s}=\{\pi \in \Pi: \operatorname{S} \perp \pi\}=\{\operatorname{S}\}^\perp.\]
\item We define $\operatorname{e}=\{\operatorname{E}\operatorname{E}\}^\perp$.
\end{enumerate}
\end{defi}  
\item We prove the following crucial theorem.
\begin{theo}\label{theo:pakstooca} Consider the $\mathcal{PAKS}$: \[(\Lambda,\Pi,\Perp,\operatorname{app},\operatorname{save},\operatorname{push}, \operatorname{K},\operatorname{S},\operatorname{cc}),\]
and the quintuple as presented in Definition \ref{defi:paksoca}: \[(A,\leq,\circ, \operatorname{k},\operatorname{s}).\] 

The above quintuple forms an $\mathcal{OCA}$. Moreover, the map
$\rightarrow$ is an implication and the element $\operatorname{e}$ is
an adjunctor and if the $\mathcal {AKS}$ is classical, so is the
$\mathcal {OCA}$.  
\end{theo}
\begin{proof} It is clear that $\circ$ is an application in $A$, that
  $\leq$ is a partial order, and we have defined the elements
  $\operatorname{k}$ and $\operatorname{s}$.  
Concerning the monotony of the application we have to prove that if: $a \supseteq a'\,,\,b \supseteq b'$, then if $\pi \in \Pi$ satisfies that $a'  \supseteq ({}^\perp b' . \pi)$, then $a \supseteq ({}^\perp b . \pi)$.

We have that $a \supseteq a' \supseteq ({}^\perp b' . \pi)$. As $b
\supseteq b'$, ${}^\perp b' \supseteq {}^\perp b$ and then $({}^\perp
b' . \pi) \supseteq ({}^\perp b . \pi)$ and the proof of the monotony
of $\circ $ is finished.  

The monotony and antimonotony of the map $\rightarrow$ is similarly proved. 
The fact that the arrow $\rightarrow$ satisfies the \emph{half
  adjunction property}: if $a \leq (b \rightarrow c)$ then $ab \leq
c$, was established in Theorem \ref{theo:adjunction}.

Next, we prove that $\operatorname{k}ab \leq a$.  
We have seen that $\operatorname{K} \in \big({}^\perp a.({}^\perp b
. a)\big)^{\perp}$ and that means that $\{\operatorname{K}\} \subseteq
\big({}^\perp a.({}^\perp b . a)\big)^{\perp}$ that implies
$\operatorname{k} \supseteq \Big({}^\perp\big({}^\perp a.({}^\perp b
. a)\big)\Big)^{\perp} \supseteq {}^\perp a.({}^\perp b . a)$. 

Now, from the above inclusion we deduce that:
$\operatorname{k}a=({}^\perp\{\pi \in \Pi: \operatorname{k} \supseteq
{}^\perp a . \pi\})^{\perp} \supseteq \{\pi \in \Pi: \operatorname{k}
\supseteq {}^\perp a . \pi\} \supseteq {}^\perp b . a$, or in other
words that $\operatorname{k}a \leq (b \rightarrow a)$ --see Definition
\ref{defi:maps}, \eqref{item:maps2}--. Using the half adjunction
property \ref{theo:adjunction}, we deduce that $\operatorname{k}ab
\leq a$.  

The condition $\operatorname{s}abc \leq (ac)(bc)$ can be proved similarly.

Indeed, it is enough to prove that $\operatorname{s}ab \leq c
\rightarrow (ac)(bc)$ that means that
$({}^\perp\{\pi:\operatorname{s}a \supseteq {}^\perp b .\pi
\})^{\perp} \supseteq \Big({}^\perp\big({}^\perp c
.(ac)(bc)\big)\Big)^{\perp}$. Then, it is enough to prove that
$\{\pi:\operatorname{s}a \supseteq {}^\perp b .\pi \} \supseteq
  {}^\perp c .(ac)(bc)$ or $\operatorname{s}a \supseteq {}^\perp b
  .{}^\perp c .(ac)(bc)$. Now, as $\operatorname{s}a=\{\pi \in
  \Pi:\operatorname{s} \supseteq {}^\perp a . \pi\}$ we have to check
  that $\operatorname{s} \supseteq {}^\perp a . {}^\perp b .{}^\perp c
  .(ac)(bc)$ or equivalently that $\operatorname{S} \perp {}^\perp a
  . {}^\perp b .{}^\perp c .(ac)(bc)$ or $\operatorname{S} \in
  {}^\perp({}^\perp a . {}^\perp b .{}^\perp c .(ac)(bc))$.    

Hence, we take $t \perp a$, $s \perp b$, $u \perp c$ and $\pi \in
(ac)(bc)$ and using Lemma \ref{lema:axiomset} (S1)
\eqref{item:circdiamond}, we deduce that $tu(su) \perp \pi$. 

Using now {\Large{\bf{\ref{item:axiomsperp}.}}} condition (S3), we
prove that $\operatorname{S} \perp t.s.u.\pi$, that is the result we
want.  

Finally the proof that $\operatorname{e}$ as introduced above
--Definition \ref{defi:paksoca}--, is an adjunctor is the content of
Theorem \ref{theo:main}.  

If we take $\operatorname{c}=\operatorname{cc}^\perp$, we proved in
Lemma \ref{lema:axiomset}, $(\mathbb{S}4)$ that
$\operatorname{cc} \in {}^\perp(((a \rightarrow b) \rightarrow
a)\rightarrow a)$, that implies that $\operatorname c \supseteq
({}^\perp(((a \rightarrow b) \rightarrow a)\rightarrow
a))^\perp=(((a\rightarrow b)\rightarrow a)\rightarrow a)$,
i.e. $\operatorname c \leq (((a\rightarrow b)\rightarrow a)\rightarrow
a)$.  
\end{proof}
\section{Construction of an $\mathcal {OCA}$ with a filter from an $\mathcal {AKS}$.} \label{section:seven}
\item Assume that we have a
  decuple \[(\Lambda,\Pi,\Perp,\operatorname{app},\operatorname{save},\operatorname{push}, 
  \operatorname{K},\operatorname{S},\operatorname{cc},\operatorname{QP}),\]
  where the first nine elements define a $\mathcal {PAKS}$ and the
  last $\operatorname{QP} \subseteq \Lambda$ is a subset of terms that
  contains the distinguished elements
  $\operatorname{K},\operatorname{S}$ and $\operatorname{cc}$ and is
  closed by application.  
\begin{defi} Define the subset $\Phi$ of $A=\mathcal P_\perp(\Pi)$ as
  follows: \[\Phi=\{f \in A: {}^\perp f \cap \operatorname{QP} \neq
  \emptyset\}=\{f \in A: \exists t \in \operatorname{QP}, t \perp
  f\}.\]  
\end{defi}
\begin{lema}\label{lema:akstoocaf} The subset $\Phi \subseteq A$ is a
  filter in $A$ --see Definition \emph{\ref{defi:filter}} --that
  contains $\operatorname{e}$ and $\operatorname{c}$.  
\end{lema}
\begin{proof}
\newcounter{ncounter}
\begin{list}{(F\arabic{ncounter})}{\usecounter{ncounter}}

\item If $f \in \Phi$ and $a \in A$ is $f \leq a$, then $a \in
  \Phi$. This is because, by hypothesis we have an element $t_f \in
      {}^\perp f \cap \operatorname{QP} \subseteq {}^\perp a \cap
      \operatorname{QP}$. Hence, $a \in \Phi$.  
\item The subset $\Phi$ is closed under application because in
  accordance with Lemma \ref{lema:axiomset}, (S1), (5), if $t_f\in
  {}^\perp f \cap \operatorname{QP}$ and $t_g\in {}^\perp g \cap
  \operatorname{QP}$ then $t_ft_g\in {}^\perp f {}^\perp g \cap
  \operatorname{QP} \subseteq {}^\perp(f \circ g) \cap
  \operatorname{QP}$.   
\item $\operatorname{k},\operatorname{s} \in \Phi$ because
  $\operatorname{K} \in {}^\perp \operatorname{k} \cap
  \operatorname{QP}$ and $\operatorname{S} \in {}^\perp
  \operatorname{s} \cap \operatorname{QP}$.  
\item $\operatorname{e} \in \Phi$ because $\operatorname{EE} \in
  {}^\perp\operatorname{e} \cap \operatorname{QP}$ --see Observation
  \ref{obse:EEQP}.\item Being
  $\operatorname{c}=\{\operatorname{cc}\}^\perp$, it is clear that:
  $\operatorname{cc} \in {}^\perp\operatorname{c} \cap
  \operatorname{QP}$.  
\end{list}
\end{proof}
\item
Now we have enough machinery in order to answer the following
question: \emph{
  Is the filter built as above closed under meets?}

Assume that it is closed under
  meets. In this case there is an element $\Omega \in \Phi$ with
  the property that 
$\Omega \leq \operatorname{k}$ and also $\Omega \leq \operatorname{s}$. 
\begin{enumerate}
\item In that situation a direct computation guarantees that $\Omega
  \Omega \operatorname{k}\operatorname{k}$ is at the same time $\Omega
  \Omega \operatorname{k}\operatorname{k} \leq k$ and $\Omega \Omega
  \operatorname{k}\operatorname{k} \leq \operatorname{sk}$. 
Indeed, $\Omega \Omega \operatorname{k}\operatorname{k}\leq
\operatorname{skkk} \leq \operatorname{kk(kk)} \leq
\operatorname{k}$. Also, $\Omega \Omega
\operatorname{k}\operatorname{k}\leq \operatorname{kskk} \leq
\operatorname{sk}$. 
\item Hence, given $f,g \in \Phi$ we have that $\Omega \Omega
  \operatorname{k}\operatorname{k}fg \leq \operatorname{k}fg \leq f$
  and also: $\Omega \Omega \operatorname{k}\operatorname{k}fg \leq
  \operatorname{sk}fg \leq \operatorname {k}g (fg) \leq g$.  

Then, in this situation for any pair $f,g \in \Phi$ the element
$\Omega \Omega \operatorname{k}\operatorname{k}fg \leq f$ and also
$\Omega \Omega \operatorname{k}\operatorname{k}fg \leq g$.    
\item Consider an $\operatorname{AKS}$ and the corresponding $\mathcal
  {OCA}$. We have that $\Omega \supseteq \operatorname{s} \cup
  \operatorname{k}$ and ${}^\perp\Omega \subseteq
  {}^\perp(\{\operatorname{S}\}^\perp \cup \{\operatorname{K}\}^\perp) 
  \subseteq {}^\perp(\{\operatorname{S}\}^{\perp}) \cap
  {}^\perp(\{\operatorname{K}\}^{\perp})$.
\item The above condition means that: $\forall Q \in {}^\perp\Omega,
  \forall \pi \in \Pi, (S\star \pi \in \Perp \Rightarrow Q \star
  \pi\in\Perp)\,\,\, \text{and}\,\,\,(\operatorname{K}\star \pi \in \Perp \Rightarrow
  Q \star \pi)\in\Perp$. 
\item Consider the $\mathcal{AKS}$ defined by the $\mathcal{KAM}$ with
  only substitutive and deterministic instructions,
  defining $\Perp=\{t\star \pi \succ \operatorname{S} \star
  \alpha\,\,\text{or} \,\, t\star \pi \succ \operatorname{K} \star
  \beta\}$, with $\alpha,\beta$ different stack constants. Since, 
  $\alpha\in \{\operatorname{S}\}^{\perp},  \beta \in
  \{\operatorname{K}\}^{\perp}$ we get 
  $\alpha, \beta \in {}^\perp\Omega$. In this situation: if $Q \in
  {}^\perp\Omega$ we get from the statement above $Q \star \alpha, Q
  \star \beta \in \Perp$, since $\alpha\in \{\operatorname{K}\}^\perp$,
  $\beta\in\{\operatorname{S}\}^\perp$. By definition of $\Perp$: $Q \star \alpha
  \succ \operatorname{S} \star \alpha$ or $Q \star \alpha \succ
  \operatorname{K} \star \beta$.  
\item Assume now that $Q \in \operatorname{QP}$. Then $q$ does not
  contain $k_\pi$ and cannot change the stack constant, and hence: $Q
  \star \alpha \succ \operatorname{S} \star \alpha$. By substitution $Q \star \beta
  \succ \operatorname{S} \star \beta$. But, again because $Q$ cannot change the stack
  constant and $q\star \beta\in\Perp$, we get $Q\star\beta\succ
  \operatorname{K}\star\beta$. Thus we obtain $\operatorname{S}\star\beta\succ \operatorname{K}\star\beta$ or
  $\operatorname{K}\star\beta\succ \operatorname{S}\star\beta$ which is impossible because both
  $\operatorname{K}\star\beta$ and $\operatorname{S}\star\beta$ does not recduce because they does
  not have arguments.  Then, we conclude that $\Omega^\perp \cap
  \operatorname{QP} = \emptyset$. This contradicts the assumption that
  $\Omega \in \Phi$.   
\end{enumerate}
A model where it is true that a pair of elements of $\Phi$ always has
a minimum is when $\Perp=\emptyset$. Here
$\operatorname{s}=\{\operatorname{S}\}^\perp=\{\operatorname{K}\}^\perp=\operatorname{k}$,
being $\operatorname{s}=\operatorname{k}$ the set $\Phi$ is a filter
in the usual sense. 
\section{From $\mathcal{OCA}$s to Tripos}
\label{section:eight}
\item Assume we have an $\mathcal {OCA}$:
  $(A,\circ,\leq,\operatorname{k},\operatorname{s})$, that is equipped
  with an implication, an adjunctor and a filter --called
  respectively: $\rightarrow\,,\,
  \operatorname{e}\,\,\text{and}\,\,\Phi$. 

Let $I$ be an arbitrary set and consider $A^I$ the cartesian product
of $I$ copies of $A$ --viewed in general as the set of functions
$A^I=\{\varphi:I \rightarrow A: \varphi\,\, \text{is a function}\}$. 
\begin{obse}\label{obse:orders}We consider some properties of the
  order and the operations in an $\mathcal{OCA}$ and its extensions to
  cartesian products.  
\begin{enumerate} 
\item We have the following orders in $A^I$. 
\begin{enumerate}  
\item \label{item:cart}{\emph{Cartesian product of $\leq$}}: If
  $\varphi,\psi \in A^I$, $\varphi \leq \psi$ if and only if $\forall
  i \in I: \varphi(i) \leq \psi(i)$. 
\item \label{item:sq} {\emph{Cartesian product of $\sqsubseteq$}}: If
  $\varphi,\psi \in A^I$, $\varphi \sqsubseteq \psi$ if and only if
  $\forall i \in I, \exists f_i \in \Phi: f_i \varphi(i) \leq
  \psi(i)$. 
\item \label{item:ntile} {\emph{Entilement order}}:\quad\quad\, If
  $\varphi,\psi \in A^I$, $\varphi \vdash \psi$ if and only if
  $\exists f \in \Phi, \forall i \in I: f \varphi(i) \leq \psi(i)$. 
\end{enumerate}
\item In the case that $\Phi$ has $\operatorname{inf}$, it is clear that the orders listed in (b) and (c) above, are equivalent. 
\item Clearly the first order above is reflexive, antisymmetric and
  transitive; the second and third orders are reflexive and
  transitive. The proof of these last properties are identical to the
  proofs of the corresponding properties of the order $\sqsubseteq$ in
  $A$. 
\item One can define the arrow in $A^I$ simply as: $(\varphi \rightarrow \psi)(i)= \varphi(i) \rightarrow \psi(i)$.
\item The ``meet'' in $A^I$ can be defined as $(\varphi \wedge \psi)(i)=
\varphi(i)\wedge \psi(i)=p\varphi(i)\psi(i)$.
\item A manner to view the entilement order is the following. Assume
  that we take $A$ to be an $\mathcal {OCA}$ as above and that $M$ is
  an $A$--module, i,e, a set $M$ together with an operation
  $(a,m)\mapsto a.m:A\times M \rightarrow M$. The standard example of
  an $A$--module is $A^I$, with the operation
  $(a.\psi)(i)=a\psi(i)$. If we have a partial orden $\leq_M \subset M
  \times M$, we can define a new order $\sqsubseteq_M \subset M \times
  M$ as follows: if $m,n \in M$ we say that $m \sqsubseteq_M n$ if and
  only if there exist an element $f \in \Phi$ such that $f.m \leq_M
  n$. In this sense the order appearing in (1)(c) above --the
  entilement order-- is obtained from the cartesian product order
  appearing in (1)(a), by the process just mentioned. 

\end{enumerate}
\end{obse} 

\item The following ``complete adjunction property --or simply
  adjunction property--'' of the order ``entile'' is important. It is
  worth noticing that it does not follow directly from the
  corresponding property proved for $A$ in Theorem
  \ref{theo:propsmain} --i.e the property valid for all $a,b,c \in A$
  thay states that $\sqleq{a \wedge b}{c} \Leftrightarrow \sqleq{a}{(b
    \rightarrow c)}$--. We need the subtler properties given in
  Corollary \ref{coro:globaladj} and Observation
  \ref{obse:converseadj}.  
\begin{theo}\label{theo:entileadjunc} In the notations above for an
  $\mathcal{OCA}$ with implication, adjunctor and filter, the
  following is true for all $\varphi, \psi, \theta  \in
  A^I$: \[\varphi \wedge \psi \vdash \theta \Longleftrightarrow
  \varphi \vdash (\psi \rightarrow \theta).\] 
\end{theo} 
\begin{proof}
\begin{enumerate}
\item [$\Longrightarrow$] Take $f \in \Phi$ such that $f(p \varphi(i)
  \psi(i))\leq \theta(i)$ for all $i \in I$. Using Corollary
  $\ref{coro:globaladj}$ we obtain that $H(f,p)\varphi(i) \leq
  (\psi(i) \rightarrow \theta(i))$. Hence, we deduce that $\varphi
  \vdash (\psi \rightarrow \theta)$. 

\item [$\Longleftarrow$] Assume that for $f \in \Phi$ we have that
  $f\varphi(i) \leq (\psi(i) \rightarrow \theta(i))$ for all $i \in
  I$. Then, in accordance with Observation \ref{obse:converseadj},
  \eqref{item:converseadj} we deduce that
  $G(f)((p\varphi(i))\psi(i))=G(f)(\varphi(i) \wedge \psi(i)) \leq
  \theta(i)$ for all $i \in I$. In other words we have proved that:
  $\varphi \wedge \psi \vdash \theta$.  
\end{enumerate}
\end{proof}
\item \label{item:enrichedoca} Next we add some structure in order to
  continue with the construction of the tripos. For an arbitrary
  subset $X \subset A$ of the $\mathcal {OCA}$,  
  there is an element $\operatorname{inf}(X) \in A$ that is the infimum
  of $X$ with respect to the order $\leq$.  
\begin{defi}\label{defi:sicomplete}
Let us consider that we have an
$\mathcal{OCA}$  
$(A,\leq,\circ, \operatorname{s,k})$, equipped with an implication
$\rightarrow$, an 
adjunctor $\operatorname{e}$ and a filter $\Phi$ as seen in
Definition \ref{defi:oca}. This    
$\mathcal {OCA}$ is said to be a $\koca$ if it is
\emph{$\operatorname{inf}$--complete}; i.e.: 
if the operator $\operatorname{inf}:\mathcal P(A) \rightarrow A$ is
everywhere defined. 
\end{defi} 
\begin{defi} We define the element $\perp \in A$ as $\perp=\operatorname{inf}A$.
\end{defi}

We list a few basic properties of the operations in the $\mathcal
{OCA}$ in relation with the element $\perp$.  
\begin{lema} \label{lema:forbeckche} Let
  us assume that $A$ is a \koca\ (c.f. Definition
  \ref{defi:sicomplete}), then:  
\begin{enumerate}
\item For all $a\in A$ we have that $\perp a = \perp$. 
\item If $b \leq (\operatorname{i} \rightarrow a)$, then
  $\operatorname{e}(\operatorname{si}b) \leq (\operatorname{i}
  \rightarrow a)$ for $a,b \in A$. In particular in the usual notation
  for the function $F$ --see Theorem \ref{theo:compoca}-- we have that
  $(F(\operatorname{e})(\operatorname{si}))(\operatorname {i}
  \rightarrow a) \leq (\operatorname {i} \rightarrow a)$.  
\item If $a \in A$, then $\operatorname{si}a\perp = \perp$. Moreover,
  for all $a,b \in A$ we have that $(F(\operatorname e) (\operatorname
  {si})) a \leq (\perp \rightarrow b)$.  
\end{enumerate} 
\end{lema}
\begin{proof}
\begin{enumerate}
\item Clearly as $\perp \leq (a \rightarrow \perp)$ we deduce that $\perp a \leq \perp$ then $\perp a = \perp$. 
\item We have that $\operatorname{si}b\operatorname{i} \leq
  \operatorname{ii}(b \operatorname{i}) \leq \operatorname{i}(b
  \operatorname{i}) \leq b \operatorname{i} \leq a$, the last equality
  coming from the hypothesis that $b \leq (\operatorname{i}
  \rightarrow a)$. 
Hence, from the basic property of the adjuntor we obtain that:
$\operatorname{e}(\operatorname{si}b) \leq (\operatorname{i}
\rightarrow a)$. If we apply the above result to the case that
$b=(\operatorname{i} \rightarrow a)$, and then Theorem
\ref{theo:compoca}, we obtain the second part of the conclusion.   
\item We have that: $\operatorname{si}a\perp \leq
  \operatorname{i}\perp(a\perp) \leq \perp (a\perp) = \perp$ where the
  first inequality comes from the characterization of
  $\operatorname{s}$ the second from the characterization of
  $\operatorname{i}$ and the third was proved in (1). Hence,
  $\operatorname{si}a\perp \leq b$ for all $b$. From the basic
  property of the adjunctor we deduce that $\operatorname e
  (\operatorname {si} a) \leq (\perp \rightarrow b)$ and the proof is
  finished proceeding in the same way than in part (2).  
\end{enumerate}
\end{proof}
\begin{obse}\label{obse:forbeckche} \begin{enumerate}
\item Notice that we have in particular proved the following assertion
  that follows directly from parts (2) and (3) of the above Lemma
  \ref{lema:forbeckche}: there is an element $g \in \Phi$ such that
  for all $a,b \in A$: $g(\operatorname{i} \rightarrow a) \leq
  (\operatorname{i} \rightarrow a)$ and $g (\operatorname{i}
  \rightarrow a) \leq (\perp \rightarrow b)$. 
\item In fact the inequality in part (3) guarantees that for all $a,b$: \[g a \leq (\perp \rightarrow b).\]  
\end{enumerate} 
\end{obse}
\begin{defi}Given $I$ a set we define the equality predicate in $A^{I \times I}$ as follows: $\operatorname{eq}_I:I \times I \rightarrow A$ 
\[\operatorname{eq}_I(i,j)=\begin{cases}\operatorname{i=skk}&\text{if}\quad i=j;\\ \perp &\text{if}\quad i \neq j.\end{cases}\] 
\end{defi}
It is clear that for all $a,b \in A$ and for all $i,j \in I$:
\begin{enumerate}
\item $\operatorname{eq}_I(i,i)a\leq a$,
\item $\operatorname{eq}_I(i,j)a \leq b$ if $i \neq j$.
\end{enumerate}
\item More can be said about $\mathcal{OCA}$s coming from $\mathcal {AKS}$s. \begin{obse}\label{obse:aksandsicomplete}
As we have seen in Sections \ref{section:six} and \ref{section:seven},
given an $\mathcal {AKS}$ we can produce an $\mathcal{OCA}$ that is
simply $A=\mathcal P_{\perp}(\Pi)$ with the order $\leq$ given by the
reverse inclusion and with a filter $\Phi$ defined as the set of
elements of $A$ that are realized by some element of the set of quasi
proofs $\operatorname{QP} \subseteq \Lambda$. The rest of the
ingredients $\circ,\rightarrow,\operatorname{s,k,e}$ are defined as before
--see in particular Theorem \ref{theo:pakstooca} and Lemma
\ref{lema:akstoocaf}.   

Notice that for this particular kind of $\mathcal{OCA}$s, \emph{both} 
$\operatorname{sup}$ and $\operatorname{inf}$ can be defined. Indeed
if $X \subset 
\mathcal P_{\perp}(\Pi)=A$, then
$\operatorname{inf}(X)=({}^\perp(\bigcup X))^{\perp}$ and
$\operatorname{sup}(X)=(^\perp(\bigcap X))^{\perp}$.    

In particular $\mathcal P_\perp(\Pi)$ is an
$\operatorname{inf}$--complete $\mathcal{OCA}$. 
\end{obse} 
\item Let $A$ be an $\mathcal {OCA}$ and we will work in the category
  denoted as
  $[\operatorname{Set}^{\text{op}},\operatorname{Preorder}]=
  \operatorname{Preorder}^{\operatorname{Set}^{\text{op}}}$,    
  that has as objects the functors $F: \operatorname{Set}^{\text{op}}
  \rightarrow \operatorname{Preord}$, and as arrows the natural
  transformations between functors. The category
  $\operatorname{Preord}$ is the category whose objects are the
  partially order sets and its arrows are the monotone functions
  between the partially ordered sets.  

\begin{defi}\label{defi:regularfunctor} Given the $\mathcal {OCA}$
  called $A$ we define the ``regular functor'' $\mathcal R_A \in
  \operatorname{Preord}^{\operatorname{Set}^{\text{op}}}$ as
  follows: \[\mathcal R_A:\operatorname{Set}^{\text{op}} \rightarrow
  \operatorname{Preord},\]  
with $\mathcal R_A(I)=(A^I,\vdash)$ where $\vdash$ is as in the
definition appearing in Observation $\ref{obse:orders}$ item
$\ref{item:entile}$. 

If $\alpha:J \rightarrow I$, then $\alpha^*=\mathcal
R_A(\alpha):(A^I,\vdash) \rightarrow (A^J,\vdash)$ is defined as:
$\alpha^*(\varphi)=\varphi \circ \alpha$.  
\end{defi}
\begin{obse} \begin{enumerate}
\item To prove that the above Definition \ref{defi:regularfunctor}
  makes sense, we have to check that $\alpha^*$ is monotone in
  relation with the order of entilement: if $\varphi,\varphi'\in A^I
  $, and $\varphi \vdash \varphi'$, then $\alpha^*(\varphi) \vdash
  \alpha^*(\varphi')$.  We have to prove that if there is an $f \in
  \Phi$ with the property that $f \varphi(i) \leq \varphi'(i)$ for all
  $i \in I$, then there is a $g \in \Phi$ such that for all $j \in J$:
  $ g\varphi(\alpha(j))=\varphi'(\alpha(j))$. This is clearly true by
  taking $f=g$. 
\item It is clear that $\mathcal R_A(\alpha \beta)=\mathcal R_A(\beta) \mathcal R_A(\alpha)$. 
\end{enumerate}
\end{obse}
Next we define another functor, with the same object part than
$\mathcal R_A$. That will be the ``right adjoint'' of $\mathcal R_A$. 
\begin{defi} Define the functor $\forall_A:\operatorname{Set}
  \rightarrow \operatorname{Preord}$. At the level of objects
  $\forall_A(I)=(A^I,\vdash)$, and for an arrow $\alpha:J \rightarrow
  I$ and $\varphi: J \rightarrow A$, we define $\forall_A
  \alpha(\varphi): I \rightarrow A$ as $\forall_A
  \alpha(\varphi)(i)=\operatorname{inf}_{j \in
    J}\{\operatorname{eq}_I(\alpha(j),i) \rightarrow \varphi(j)\}$
  with $i \in I$.  
\end{defi}
\begin{obse}\label{obse:makesense} 
  We observe first that the definition above makes sense: we want
  to show that if $\varphi \vdash \varphi'$ for $\varphi,\varphi'\in
  A^J$, then $\forall_A \alpha(\varphi) \vdash \forall_A
  \alpha(\varphi')$.  In other words, if there is an $f \in \Phi$ such
  that for all $j \in J$, $f\varphi(j) \leq \varphi'(j)$, then there
  exists a $g \in \Phi$ such that $g (\operatorname{inf}_{j \in
    J}\{\operatorname{eq}_I(\alpha(j),i) \rightarrow \varphi(j))\}
  \leq \operatorname{inf}_{j \in J}\{\operatorname{eq}_I(\alpha(j),i)
  \rightarrow \varphi'(j)\}$ for all $i\in I$. 
  Using the fact that $\rightarrow$ is monotone in the second variable
  we have that: $\operatorname{inf}_{j \in
    J}\{\operatorname{eq}_I(\alpha(j),i) \rightarrow \varphi'(j)\} \geq
  \operatorname{inf}_{j \in J}\{\operatorname{eq}_I(\alpha(j),i)
  \rightarrow f\varphi(j)\}$. Using Corollary \ref{coro:globaladj},
  \eqref{item:particular} we deduce that for some $g \in \Phi$ --in fact in
  accordance with the mentioned corollary, $g=F(e)F(f) \in \Phi$-- we
  have that $\operatorname{inf}_{j \in
    J}\{\operatorname{eq}_I(\alpha(j),i) \rightarrow f\varphi(j)\} \geq
  g\operatorname{inf}_{j \in J}\{\operatorname{eq}_I(\alpha(j),i)
  \rightarrow \varphi(j)\}$.  
  Putting both inequalities together we deduce that
  \[g\operatorname{inf}_{j \in J}\{\operatorname{eq}_I(\alpha(j),i)
  \rightarrow \varphi(j)\} \leq \operatorname{inf}_{j \in J}\{ 
  \operatorname{eq}_I(\alpha(j),i) \rightarrow \varphi'(j)\},\] that is our conclusion.
\end{obse}
Next we prove that for an arbitrary $\alpha:J \rightarrow I$ the map
$\forall_A(\alpha):A^J \rightarrow A^I$ is a ``right adjoint'' of
$\alpha^*=\mathcal R_A(\alpha):A^I \rightarrow A^J$ with respect to
the orden $\vdash$.  
\begin{theo}\label{theo:adjparatodo} Assume that $A$ is a \koca. 
If $I,J \in \operatorname{Set}$, $\alpha:J \rightarrow I$ is a
function and $\varphi \in A^J\,,\, \psi \in A^I$, then: 
\[\alpha^*(\psi) \vdash \varphi \Leftrightarrow \psi \vdash
\forall_A\alpha (\varphi).\] 
\end{theo}
\begin{proof}
\begin{enumerate}
\item[$\Longrightarrow$] From the hypothesis, we deduce that there is
  an element $f \in \Phi$, with the property that for all $j \in J$ $f
  \psi(\alpha(j)) \leq \varphi(j)$.  
We take a general $i \in I$, and prove first that for all $i,j$ we have that: 
$E(f) \psi(i) \operatorname{eq}_I(\alpha(j),i) \leq \operatorname{i}(f \psi(\alpha(j))) \leq f \psi(\alpha(j)) \leq \varphi(j)$. 
\begin{itemize}
\item If $i \neq \alpha(j)$ we deduce from 
Theorem \ref{theo:compoca},\eqref{eqn:cero} and Lemma \ref{lema:forbeckche} that in this situation
$E(f) \psi(i) \operatorname{eq}_I(\alpha(j),i) \leq \perp(f
\psi(i))=\perp \leq \varphi(j)$.  
\item If $i = \alpha(j)$, we deduce similarly that $E(f) \psi(i)
  \operatorname{eq}_I(\alpha(j),i) \leq \operatorname{i}(f
  \psi(\alpha(j))) \leq f \psi(\alpha(j)) \leq \varphi(j)$.  
\end{itemize}
Hence, using the basic property of the adjunctor, we see that:
$\operatorname{e}(E(f) \psi(i)) \leq (\operatorname{eq}_I(\alpha(j),i)
\rightarrow \varphi(j))$.  Using as before Theorem
\ref{theo:compoca},\eqref{eqn:primera}, we obtain that
$(F(\operatorname{e})E(f)) \psi(i) \leq
(\operatorname{eq}_I(\alpha(j),i) \rightarrow \varphi(j))$, and taking
$\operatorname{inf}_j$ we deduce that if we call
$g=F(\operatorname{e})E(f) \in \Phi$,  
we have that:
\[ g \psi(i) \leq \forall_A\alpha\varphi(i) \text{ for all $i\in
  I$, \quad i.e.}\quad \psi
\vdash \forall_A\alpha(\varphi).\]  
  
\item[$\Longleftarrow$] Our hypothesis guarantees the existence of an
  element $f \in \Phi$ such that for all $i,j$ we have: $f.\psi(i)
  \leq (\operatorname{eq}_I(\alpha(j),i) \rightarrow \varphi(j))$. In
  particular if $i=\alpha(j)$ we have that for all $j \in J$,
  $f.\psi(\alpha(j)) \leq (\operatorname{i} \rightarrow \varphi(j))$
  and then, by the basic (half) adjunction condition we see that
  $(f \psi(\alpha(j)))\operatorname{i} \leq \varphi(j)$. Using Theorem
  \ref{theo:compoca},\eqref{eqn:cuarta}, we obtain that:
  $M(f,\operatorname{i})\psi(\alpha(j)) \leq \varphi(j)$ with
  $M(f,\operatorname{i}) \in \Phi$ or in other words, we obtain that
  for all $j \in J$, $M(f,\operatorname{i})\alpha^*(\psi)(j) \leq
  \varphi(j)$ that is what we wanted to conclude.   
\end{enumerate}
\end{proof}
\item We want to prove the so called theorem of Beck--Chevalley. 
\begin{theo} Assume that $A$ is a \koca\ and that the following is a
  pull back diagram in the category of sets:  
\[
\xymatrix{P\ar[r]^{\rho}\ar[d]_{\pi}&J\ar[d]^{\alpha}\\
K\ar[r]_{\beta}&I}
\]
and consider the corresponding diagram that follows:
\[
\xymatrix{A^P\ar[d]_{\forall \pi}&A^J\ar[d]^{\forall \alpha}\ar[l]_-{\rho^*}\\
A^K&A^I\ar[l]^-{\beta^*}}
\]

Then, the second diagram commutes in the sense that for all $\varphi \in A^J$: 
\[\beta^*(\forall \alpha \varphi) \vdash \forall
\pi(\rho^*(\varphi))\quad \text{and} \quad  \forall \pi(
\rho^*(\varphi)) \vdash \beta^*(\forall \alpha \varphi).\] 
\end{theo}
\begin{proof}
\begin{enumerate}
\item The proof that $\beta^*(\forall \alpha (\varphi)) \vdash \forall
  \pi( \rho^*(\varphi))$ follows from general categorical properties.
  We start with the counit relation in Theorem \ref{theo:adjparatodo}
  that guarantees that $\alpha^*\forall\alpha(\varphi) \vdash \varphi$
  and applying $\rho^*$ deduce that
  $\rho^*\alpha^*\forall\alpha(\varphi) \vdash \rho^*(\varphi)$. From
  the functoriality of $\mathcal R$ we obtain that
  $\pi^*\beta^*\forall\alpha(\varphi) \vdash \rho^*(\varphi)$, and by
  Observation \ref{obse:makesense} we get: $\forall
  \pi\pi^*\beta^*\forall\alpha(\varphi) \vdash \forall
  \pi\rho^*(\varphi)$. 
  Finally, using the unit of the adjunction in Theorem
  \ref{theo:adjparatodo} we conclude that $\beta^*\forall\alpha(\varphi)
  \vdash \forall \pi\pi^*\beta^*\forall\alpha(\varphi) \vdash \forall
  \pi\rho^*(\varphi)$. 
\item Now we prove that $(\forall \pi( \rho^*(\varphi)) \vdash \beta^*(\forall \alpha \varphi))$.
We fix $k_0 \in K$ and need to find an element $g \in \Phi$ such that for all $j \in J$ we have that:
\[g \operatorname{inf}_{z \in P}\{\operatorname{eq}_K(\pi(z),k_0)
\rightarrow \varphi (\rho(z))\} \leq
\big(\operatorname{eq}_I(\alpha(j),\beta(k_0)) \rightarrow
\varphi(j)\big).\]  

We distinguish two possiblities considering if there is an element  $z_0 \in P$ such that $\pi(z_0)=k_0$ or not.
\begin{itemize}
\item Suppose that we take $z_0 \in P$ with the property that $\pi(z_0)=k_0$, i.e. $z_0 \in \pi^{-1}(k_0)$.
In this situation $\operatorname{eq}_K(\pi(z_0),k_0) \rightarrow
\varphi (\rho(z_0)) = \operatorname{i} \rightarrow \varphi
(\rho(z_0))$ and it follows from Observation \ref{obse:forbeckche} and
using the notation there, that
$g\big(\operatorname{eq}_K(\pi(z_0),k_0) \rightarrow \varphi
(\rho(z_0))\big)  \leq \big(\operatorname{eq}_K(\pi(z_0),k_0)
\rightarrow \varphi (\rho(z_0))\big)$

Now, given an arbitrary $j \in K$ it may happen that $\rho(z_0)=j$ or
$\rho(z_0) \neq j$. In the first case we have that
$\alpha(j)=\beta(k_0)$ and that means that
$\operatorname{eq}_I(\alpha(j),\beta(k_0)) \rightarrow
\varphi(j)=\operatorname{i} \rightarrow \varphi(\rho(z_0))$.  
Hence in this case we have that $g (\operatorname{eq}_K(\pi(z_0),k_0)
\rightarrow \varphi (\rho(z_0))) \leq
\big(\operatorname{eq}_I(\alpha(j),\beta(k_0)) \rightarrow
\varphi(j)\big)$. 
Otherwise, if $\rho(z_0) \neq j$, we cannot have that
$\alpha(j)=\beta(k_0)$ as can be deduced by the basic properties of
the pull back.  
Hence, we have that $\operatorname{eq}_I(\alpha(j),\beta(k_0))
\rightarrow \varphi(j)= \perp \rightarrow \varphi(j)$ and we obtain
again that $g (\operatorname{eq}_K(\pi(z_0),k_0) \rightarrow \varphi
(\rho(z_0))) \leq \big(\operatorname{eq}_I(\alpha(j),\beta(k_0))
\rightarrow \varphi(j)\big)$ from Observation \ref{obse:forbeckche}
where we proved that $g (\operatorname{i} \rightarrow \varphi
(\rho(z_0))) \leq (\perp \rightarrow c)$ for all $c \in A$. 

Hence we have that for all $j \in J$,\\ 
\quad$g \operatorname{inf}_{z \in P}\{\operatorname{eq}_K(\pi(z),k_0)
\rightarrow \varphi (\rho(z))\} \leq
g(\operatorname{eq}_K(\pi(z_0),k_0) \rightarrow \varphi (\rho(z_0)))
\leq \operatorname{eq}_I(\alpha(j),\beta(k_0)) \rightarrow
\varphi(j)$. 

\item Suppose that $ \emptyset = \pi^{-1}(k_0) \subseteq P$. In that
  case is clear that there is no pair $(j,k_0) \in J \times K$ such
  that $\alpha(j)=\beta(k_0)$ --this follows directly from the fact
  that the diagram of sets is a pullback. 
Hence, the inequality to be proved states that for all $j \in J$:  
\[g \operatorname{inf}_{z \in P}\{\perp \rightarrow \varphi
(\rho(z))\} \leq \big(\perp \rightarrow \varphi(j)\big).\] 
The validity of these type of inequalities is the content of
Observation \ref{obse:forbeckche}, (2).  
\end{itemize}
\end{enumerate}
\end{proof}

\item We want to prove the existence of a generic predicte.
\begin{defi} Let $A$ be a \koca\footnote{Observe that for this
    definition and for the theorem that follows, the
    $\operatorname{inf}$--completeness of $A$ is unnecessary.}. The
  maps of the form 
  $\mathcal{R}_A(\alpha): A^I \rightarrow 
  A^J$ for $\alpha:J \rightarrow I$ are called \emph{reindexing}
  maps. 

A pair $(T, \Sigma)$ with $T \in A^{\Sigma}$ is called a \emph{generic
  predicate} if for all pairs $(\varphi, I)$ with $I \subset A$ and
$\varphi \in A^I$, there is a morphism $\alpha: I \rightarrow \Sigma$ such that
$\alpha^*(T)=\varphi$. 
\end{defi}

\begin{theo} In the context of an $\operatorname{inf}$--complete
  $\mathcal{OCA}$, a generic predicate exists.  
\end{theo}
\begin{proof} Just take $\Sigma=A$ and $T \in A^A$ the identity map
  $T=\operatorname{id}_A:A \rightarrow A$. It is clear that if
  $\varphi:I \rightarrow A$, then $\varphi^*(T)=\operatorname{id}_A
  \circ \varphi=\varphi$. 
\end{proof}
\section{Internal realizability in $\koca$s}\label{section:nine}
\item
We have shown that the class of ordered combinatory algebras that,
besides a filter of distinguished truth values are equipped with an
implication, an adjunctor and satisfy a completeness condition with
respect to the infimum over arbitrary subsets -- i.e.: 
\koca s-- is rich enough as to allow the Tripos construction and as such
its objects can be taken as the basis of the categorical perspective
on classical realizability --\`a la Streicher--. In this section we
show that we can define realizability in this type of combinatory
algebras, and thus, to define realizability in high order arithmetic.  

\begin{defi}\label{defi:LomegaLang}
Consider a set of constants of kinds, one of its elements is denoted
by $o$. The language of kinds is given by the following grammar:
\[
\sigma, \tau :: = c\quad |\quad \sigma \to \tau
\]
Consider an infinite set of variables labelled by kinds
$x^\tau$. Suppose that 
we have infinitely many variables labelled of the kind $\tau$ for each
kind $\tau$. Consider also a set of constants $a^\tau, b^\sigma,
\dots$ labelled with a kind. The language $\mathcal{L}^\omega$ of
order $\omega$ is 
defined by the following grammar:
\[
M^\sigma, N^{\sigma\to\tau}, A^o, B^o :: = x^\sigma\quad |\quad
a^\sigma\quad |\quad (\lambda
x^\sigma.M^\tau)^{\sigma\to\tau}\quad |\quad (N^{\sigma\to\tau}
M^\sigma)^\tau\quad |\quad (A^o\Rightarrow B^o)^o\quad |\quad (\forall
x^\tau.A^o)^o
\] 
$o$ represents the type of truth values. The expressions labelled by
$o$ are called ``formul\ae". The symbols $\to$ and $\Rightarrow$, when
itereted, are associated on the right side. On the other hand, the
application, when iterated, are associated on the left side. 
\end{defi}
\begin{defi}\label{defi:LomegaTyp} Consider a $\mathcal{^KOCA}$ $A$
  and a set of variables $\mathcal{V}=\{x_1, x_2, \dots\}$. A
  declaration is a string of the shape $x_i:A^o$.  
  A context is a string of the shape $x_1:A_1^o, \dots,
x_k:A_k^o$, i.e.: contexts are finite sequences of declarations. The
contexts will be often denoted by capital greek letters: 
$\Delta, \Gamma, \Sigma$. A sequent is a string of the shape
$x_1:A_1^o, \dots, 
x_k:A_k^o\vdash p:B^o$ where $p$ is a polynomial of $A[x_1, \dots,
  x_k]$. The left side of a sequent is a context. When we do
not explicite the declarations of the context of a sequent, we 
will write it as~$\Gamma\vdash p:B^o$. Typing rules are trees of the shape  
\begin{prooftree}
\AxiomC{$S_1$}
\AxiomC{$\dots$}
\AxiomC{$S_h$}
\RightLabel{\scriptsize(Rule)}
\TrinaryInfC{$S_{h{+}1}$}
\end{prooftree}
where $h\geq 0$ and $S_1, \dots, S_{h{+}1}$ are sequents. The typing
rules for $\mathcal{L}^\omega$ are the following:
\begin{flushright}
  \begin{prooftree}
    \AxiomC{}
    \RightLabel{\scriptsize(ax)}
    \LeftLabel{\scriptsize{(where $x_i:A^o_i$ appears in $\Gamma$)}}
    \UnaryInfC{$\Gamma\vdash x_i:A^o_i$}
  \end{prooftree}
  \begin{prooftree}
    \AxiomC{$\Gamma, x:A^o\vdash p:B^o$}
    \RightLabel{\scriptsize{$(\to_i)$}}
    \UnaryInfC{$\Gamma\vdash e(\lambda^*x\ p):(A^o\Rightarrow B^o)^o$} 
  \end{prooftree}
  \begin{prooftree}
    \AxiomC{$\Gamma\vdash p:(A^o\Rightarrow B^o)^o$}
    \AxiomC{$\Gamma\vdash q:A^o$}
    \RightLabel{\scriptsize{$(\to_e)$}}
    \BinaryInfC{$\Gamma\vdash pq:B^o$}
  \end{prooftree}
  \begin{prooftree}
    \AxiomC{$\Gamma\vdash p:A^o$}
    \RightLabel{\scriptsize{$(\forall_i)$}}
    \LeftLabel{\scriptsize{(where $x^\sigma$ does not appears free in $\Gamma$)}}
    \UnaryInfC{$\Gamma\vdash p:(\forall x^\sigma A^o)^o$}
  \end{prooftree}
  \begin{prooftree}
    \AxiomC{$\Gamma\vdash p:(\forall x^\sigma A^o)^o$}
    \RightLabel{\scriptsize{$(\forall_e)$}}
    \UnaryInfC{$\Gamma\vdash p:(A^o\{x^\sigma:=M^\sigma\})$}
  \end{prooftree}
\end{flushright}
\end{defi} 
\begin{defi}\label{defi:LomegaSem} 
  Let us consider $\mathcal{A}=(A,\leq, \circ, \operatorname{s},
  \operatorname{k}, \to, \operatorname{e}, \Phi, 
  \operatorname{inf})$\footnote{At this point we must be more precise
    and distinguish notationally the $\mathcal{OCA}$ \ $\mathcal{A}$
    from its underlying set $A$.} a 
  complete $\mathcal{^KOCA}$. We define the interpretation of
  $\mathcal{L}^\omega$ as follows:
  \begin{enumerate}
    \item For \emph{kinds}: The interpretation of a constant $c$ is a set
      $\llbracket c \rrbracket$. In particular, the constant $o$ is
      interpreted as the underlying set of $\mathcal{A}$, i.e.:
      $\llbracket o\rrbracket = A$. Given two kinds $\sigma, \tau$,
      the interpretation $\llbracket \sigma \to \tau\rrbracket$ is the
      space of functions $\llbracket \tau\rrbracket^{\llbracket \sigma
      \rrbracket}$ 
    \item For \emph{expressions}: In order to interpret expressions,
      we start choosing an assignment $\mathfrak{a}$ for the
      variables 
        $x^\sigma$ such that $\mathfrak{a}(x^\sigma)\in\llbracket
        \sigma\rrbracket$. As it is usual in semantics, the
        substitution-like 
        notation $\{x^\sigma:=s\}$ affecting an assignment
        $\mathfrak{a}$ 
        modifies it by redefining 
        $\mathfrak{a}$ over $x^\sigma$ as the statement   
        $\mathfrak{a}\{x^\sigma:=s\}(x^\sigma):= s$. We proceed similarly for
        interpretations. 
      \begin{itemize}
      \item For an expression of the shape $x^\sigma$, its
        interpretation is $\llbracket x^\sigma\rrbracket =
        \mathfrak{a}(x^\sigma)$. 
      \item For an expression of the shape $\lambda x^\sigma M^\tau$,
        its interpretation is the function 
        $\llbracket \lambda x^\sigma M^\tau\rrbracket\in\llbracket
        \sigma \to \tau\rrbracket$ defined as $\llbracket
        \lambda x^\sigma M^\tau\rrbracket(s):= \llbracket
        M^\tau\rrbracket\{x^\sigma:=s\}$ for all $s\in\llbracket \sigma
        \rrbracket$. 
      \item For an expression of the shape $(N^{\sigma\to
        \tau}M^\sigma)^\tau$ its interpretation is $\llbracket
        (N^{\sigma\to
        \tau}M^\sigma)^\tau\rrbracket:=\llbracket N^{\sigma\to
        \tau}\rrbracket 
        \big(\llbracket M^\sigma \rrbracket\big) $
      \item For an expression of the shape $(A^o\Rightarrow B^o)^o$
        its interpretation is $\llbracket (A^o\Rightarrow
        B^o)^o\rrbracket:= \llbracket A^o\rrbracket \to \llbracket
        B^o\rrbracket$.   
      \item For an expression of the shape $(\forall x^\sigma A^o)^o$
        its interpretation is \[\llbracket (\forall x^\sigma
        A^o)^o\rrbracket:= \operatorname{inf}\big\{\llbracket
        A^o\rrbracket\{x^\sigma:=s\}\ \big|\ s\in\llbracket
        \sigma\rrbracket\big\}\]
      \end{itemize}
  \end{enumerate}
  We say that  
  $\mathcal{A}$ satisfies a sequent $x_1:A_1^o, \dots,
  x_k:A^k\vdash p:B^o$ if and only if for all assignment
  $\mathfrak{a}$ and for all~$b_1, \dots, b_k\in A$,
  if $b_1 \leq \llbracket
  A_1^o\rrbracket, \dots, b_k\leq \llbracket A^o_k\rrbracket$
  then $p\{x_1:=b_1, \dots, x_k:=b_k\}\leq \llbracket
  B^o\rrbracket$. In this case we write that:~$\mathcal{A}\models x_1{:}A_1^o, \dots,
  x_k{:}A^k\vdash p{:}B^o$.  

  A rule: \begin{prooftree}
\AxiomC{$S_1$}
\AxiomC{$\dots$}
\AxiomC{$S_h$}
\RightLabel{\scriptsize(Rule)}
\TrinaryInfC{$S_{h{+}1}$}
\end{prooftree}
is said to be \emph{adequate} if and only if for every $\mathcal{A}\in\mathcal{^KOCA}$, if
$\mathcal{A}\models S_1, \dots, S_h$ then $\mathcal{A}\models
S_{h{+}1}$.   
\end{defi} 
\begin{theo}\label{theo:LomegaAdeq}
  The rules of the typing system appearing in Definition \ref{defi:LomegaTyp}, are adequate.  
\end{theo}
\begin{proof}
  For {\scriptsize{(ax)}} is evident. 

  For the implication rules:
  \begin{itemize}
    \item[{$(\to)_i$}] Assume $\mathcal{A}\models \Gamma,
      x:A^o\vdash p:B^o$ where $\Gamma=x_1:A_1^o, \dots,
      x_k:A_k^o$. Consider an assignment $\mathfrak{a}$ and~$b_1, \dots, b_k\in A$ such that $b_i\leq
      \llbracket A^o_i\rrbracket$. We get:
      \begin{center}
      \begin{tabular}{rcl}
      $(\lambda^*x p)\{x_1:=b_1,
      \dots, x_k:=b_k\}\llbracket
      A^o\rrbracket$&$=$&$(\lambda^*x p\{x_1:=b_1,
      \dots, x_k:=b_k\})\llbracket A^o\rrbracket\leq$\\
      &&$p\{x_1:=b_1,
      \dots, x_k:=b_k, x:=\llbracket A^o\rrbracket\}\leq$\\
      &&$\llbracket B^o\rrbracket$
      \end{tabular}
      \end{center}
      the last inequality by the assumption $\mathcal{A}\models
      \Gamma, x:A^o\vdash p:B^o$. 

      Applying the adjunction property we
      deduce that~$\operatorname{e}(\lambda^*x p)\{x_1:=b_1,
      \dots, x_k:=b_k\}\leq \llbracket (A^o\Rightarrow
      B^o)^o\rrbracket$. Since the above is valid for all the
      assignments, we conclude $\mathcal{A}\models \Gamma\vdash 
      \operatorname{e}(\lambda^*x\ p):(A^o\Rightarrow B^o)^o$.

      \item[{$(\to)_e$}] Assume $\mathcal{A}\models
        \Gamma\vdash 
      p:(A^o\Rightarrow B^o)^o$ and $\mathcal{A}\models
      \Gamma\vdash q: A^o$ where
      $\Gamma=x_1:A_1^o, \dots, x_k:A_k^o$. Consider an assignment
      $\mathfrak{a}$ and~$b_1, \dots,
      b_k\in A$ such that $b_i\leq 
      \llbracket A^o_i\rrbracket$. By hypothesis we get:
      \begin{center}
      \begin{tabular}{rcl}
        $p\{x_1:=b_1,
        \dots, x_k:=b_k\}$&$\leq$&$\llbracket A^o\rrbracket \to
        \llbracket B^o\rrbracket$\\
        &and&\\
        $q\{x_1:=b_1,
        \dots, x_k:=b_k\}$&$\leq$&$\llbracket A^o\rrbracket$\\
      \end{tabular}
      \end{center}
      and by monotonicity of the application in $\mathcal{A}$ we obtain:
        \[pq\{x_1:=b_1,
        \dots, x_k:=b_k\}\quad \leq\quad (\llbracket A^o\rrbracket \to
        \llbracket B^o\rrbracket)\quad \llbracket
        A^o\rrbracket\quad \leq\quad\llbracket B^o\rrbracket\] 
          Since the above is valid for all the assignments, we
          conclude that $\mathcal{A}\models \Gamma\vdash 
          pq:\llbracket B^o\rrbracket$.  
  \end{itemize}
  For the quantifiers:
  \begin{itemize}
    \item[{$(\forall)_i$}] Assume $\mathcal{A}\models \Gamma\vdash
      p:A^o$ and that $x^\sigma$ does not appear free in $\Gamma$, where
      $\Gamma=x_1:A_1^o, \dots, x_k:A^o_k$. Consider an
      assignment~$\mathfrak{a}$ and~$b_1, \dots, 
      b_k\in A$ such that $b_i\leq \llbracket A^o_i\rrbracket$.  

      Since $A_1^o, \dots, A_k^o$ does not depend upon $x^\sigma$, by
      the assumption $\mathcal{A}\models \Gamma\vdash
      p:A^o$, we get: 
      \begin{center}$p\{x_1:=b_1, \dots, x_k:=b_k\}\leq\llbracket
        A^o\rrbracket\{x^\sigma:=s\}$ for all $s\in\llbracket
        \sigma\rrbracket$
      \end{center} 
      Then $p\{x_1:=b_1, \dots, x_k:=b_k\}\leq\operatorname{inf}
      \{\llbracket A^o\rrbracket\{x^\sigma:=s\}\ |\ s\in\llbracket
      \sigma\rrbracket\}=\llbracket (\forall x^\sigma
      A^o)^o\rrbracket$. We conclude as before that
      $\mathcal{A}\models \Gamma\vdash 
      p:(\forall x^\sigma A^o)^o$.
      \item[$(\forall)_e$] Assume $\mathcal{A}\models \Gamma\vdash
      p:(\forall x^\sigma A^o)^o$, where $\Gamma=x_1:A_1^o, \dots,
      x_k:A^o_k$. Consider an assignment $\mathfrak{a}$ and~$b_1,
      \dots, b_k\in A$ such that $b_i\leq \llbracket
      A^o_i\rrbracket$. By the assumption $\mathcal{A}\models
      \Gamma\vdash p:(\forall x^\sigma A^o)^o$ we get: 
      \[p\{x_1:=b_1,
      \dots, x_k:=b_k\}\leq \llbracket A^o\rrbracket\{x^\sigma:=s\}
      \mbox{ for all }s\in\llbracket \sigma\rrbracket\] 
      Since 
      $\llbracket M^\sigma\rrbracket\in\llbracket \sigma\rrbracket$ we
      obtain:
      \[p\{x_1:=b_1, \dots, x_k:=b_k\}\leq \llbracket A^o\rrbracket
      \{x^\sigma:=\llbracket M^\sigma\rrbracket \}= 
      \llbracket A^o\{x^\sigma:=M^\sigma\}\rrbracket\]
      We conclude as before that 
      $\mathcal{A}\models \Gamma\vdash p:A^o\{x^\sigma:=M^\sigma\}$.  
  \end{itemize}
\end{proof}
The language of high order Peano Arithmetics
--$(\operatorname{PA})^\omega$--is an instance of  
$\mathcal{L}^\omega$ where we distinguish a constant of kind $I$ and two
constants of expression $0^I$ and $\operatorname{succ}^{I\to
  I}$. 
\begin{defi} For each kind $\sigma$ we define the Leibniz equality
  $=_\sigma$ as follows:
\[
x_1^\sigma =_\sigma x_2^\sigma :\equiv \forall y^{\sigma\to o} \Big
((y^{\sigma \to o} x_1^\sigma)^o \Rightarrow (y^{\sigma \to
  o}x_2^\sigma)^o\Big)^o  
\]
\end{defi}
The axioms of Peano Arithmetics are equalities over the kind $I$,
except for $\forall x^I ((\operatorname{succ}^{I\to I}x^I =_I
0^I)\Rightarrow \bot)^o$ --which we abbreviate $\forall x^I
(\operatorname{succ}^{I\to I}x^I\neq 0^I)^o$-- and for the induction principle. 

From the work of Krivine (c.f.: \cite{kn:kr2003}) we can conclude that all Peano Axioms
except the induction principle are realized in every \koca.
\begin{lema}
All equational axioms of Peano Arithmetics are realized in every
\koca. 
\end{lema}
\begin{proof}
 For the axioms which are 
equalities, the identity term $\lambda^*x x$ suffices as a realizer. For the axiom $\forall x^I
(\operatorname{succ}^{I\to I}x^I\neq 0^I)$, the term $\lambda^* x\ x\operatorname{s}$
is a realizer. 
\end{proof}
\begin{defi}
The formula $\mathds{N}(z^I)$ is defined as: 
\[
\forall x^{I \to o}(\forall y^I ((x^{I\to o}y^I)^o\Rightarrow (x^{I\to
  o}(\operatorname{succ}^{I\to I}y^I))^o\Rightarrow ((x^{I\to o}0^I)^o
\Rightarrow (x^{I\to o}z^I)^o)^o
\]
\end{defi} 
The meaning of this definition is that $\mathds{N}(z^I)$ is satisfied in the sort
$I$ by the individuals $z^I$ which are in all the inductive
sets. 

\end{list}

\end{document}